\documentclass{amsart}
\usepackage[english,german]{babel}
\usepackage{amsmath}
\usepackage{amsfonts}
\usepackage{amssymb}
\usepackage{amsthm}
\usepackage{amsopn}
\usepackage{amscd}
\usepackage{amsxtra}

\usepackage{mdwlist}
\usepackage{subfigure}
\usepackage{multicol}
\usepackage{pstricks}
\usepackage[utf8]{inputenc}
\usepackage{courier}
\usepackage{url}
\usepackage{xypic}
\usepackage[all,cmtip]{xy}

\usepackage[color,final]{showkeys}
\definecolor{labelkey}{rgb}{1,0,1}
\definecolor{refkey}{gray}{1}

\usepackage[vcentering,dvips]{geometry}

\newtheorem{theorem}{Theorem}[section]
\newtheorem{lemma}[theorem]{Lemma}
\newtheorem{proposition}[theorem]{Proposition}

\newtheorem{corollary}[theorem]{Corollary}

\theoremstyle{definition}

\newtheorem{conjecture}[theorem]{Conjecture}

\theoremstyle{remark}
\newtheorem{example}[theorem]{Example}
\newtheorem*{example*}{Example}

\numberwithin{equation}{section}

 \def\Q{\mathbb{Q}}

\def\P{\mathbb{P}}

\def\cA{\mathcal{A}}
 \def\cB{\mathcal{B}}
\def\cC{\mathcal{C}}
\def\cD{\mathcal{D}}

 \def\cK{\mathcal{K}}

\def\a{\alpha}

\def\d{\delta} 
\def\D{\Delta}

\def\l{\lambda} 
\def\m{\mu}

\def\M{\overline{M}}

\def\1tn{\{1, \ldots, n\}}

\def\cA{\mathcal{A}}

\newgray{gray1}{.65}
\newgray{gray2}{.75}
\newgray{gray3}{.85}

\title[Fulton's conjecture for $\overline{M}_{0,7}$]%
 {Fulton's conjecture for $\overline{M}_{0,7}$} %
 
\author{Paul Larsen}

\address{Humboldt-Universit\"at zu Berlin, Institut f\"ur Mathematik, 10099 Berlin, Germany}
\email{larsen@mathematik.hu-berlin.de}

\begin{document}
\selectlanguage{english}
\begin{abstract}
Fulton's conjecture for the moduli space of stable pointed rational curves, $\M_{0,n}$, claims that a divisor non-negatively intersecting all $F$-curves is linearly equivalent to an effective sum of boundary divisors. Our main result is a proof of Fulton's conjecture for $n=7$. A key ingredient in the proof is an $\binom{n}{4}$-dimensional subspace of the N\'eron-Severi space of $\M_{0,n}$, defined by averages of Keel relations, for which we prove Fulton's conjecture for all $n$.
\end{abstract}
\maketitle

 \section{Introduction}
 \label{secF:introduction}
 A central open problem concerning the birational geometry of the moduli space of stable pointed rational curves, $\M_{0,n}$, is the $F$-conjecture, which posits that the Mori cone of $\M_{0,n}$ is generated by a finite collection of rational curves called $F$-curves (defined below). This conjecture has been proven for $n\leq 7$ in \cite{km} using techniques from the minimal model program and negativity properties of the canonical bundle that do not hold for higher $n$.
 
It was realized in \cite{MR1887636} that the $F$-conjecture is implied
by another conjecture that can be stated in terms of convex geometry
of finite dimensional vector spaces. This conjecture, which we call
Fulton's conjecture for divisors, first appeared in \cite{km}, and was proven for $n \leq 6$ in
\cite{fa}, and, independently and by different methods, in
\cite{MR1938752}. The main result of this paper is a proof of Fulton's
conjecture for $n =7$. We begin by describing the usual formulation of the conjecture.

An element of $\M_{0,n}$ is a tree of projective lines with at least three \emph{special} points (i.e. marked points or nodes) on each component, modulo automorphism (see \cite{MR1034665} or \cite{MR702953}). There is a stratification of $\M_{0,n}$ by topological type as follows: a \emph{codimension 1-stratum} is an irreducible component of the locus of points of $\M_{0,n}$ having at least one node. Hence the generic element of a codimension 1-stratum has two irreducible components, with some $J \subseteq \1tn$, $2 \leq |J|\leq n-2$, giving the marked points on one component, and $J^c$ giving the marked points on the other. The resulting divisor is called a \emph{boundary divisor}, and is denoted by $\D_J$, $\D_{J^c}$, or $\D_{J, J^c}$. We continue increasing the number of nodes of the general element until reaching the \emph{dimension 1-strata}, which are irreducible components of loci with at least $n-4$ nodes. Any curve in $\M_{0,n}$ numerically equivalent to a dimension 1-stratum is called a \emph{Faber}- or \emph{F-curve}.

This stratification led Fulton to ask if the effective $k$-cycles of $\M_{0,n}$ were generated by the $k$-strata, as is the case with toric varieties. The question has a negative answer for divisors (see \cite{MR1882122}), but by restricting to divisors non-negatively intersecting all $F$-curves, called \emph{$F$-nef divisors}, we obtain

\begin{conjecture}[Fulton's conjecture]
\label{conj:Falggeom}
 Every $F$-nef divisor is numerically equivalent to an effective sum of boundary divisors.
\end{conjecture}

To interpret Fulton's conjecture in terms of cones in
finite-dimensional vector spaces, let $V$ be the $2^{n-1} - n  -1$
dimensional vector space over $\Q$ with standard basis elements
labeled $\D_{J, J^c}$, where $J \subseteq \1tn$, and $|J|, |J^c| \geq
2$. Since the numerical equivalence classes of boundary divisors
generate the N\'eron-Severi space, $N^1(\M_{0,n})_{\Q}$, there is a
surjection $\phi: V \to N^1(\M_{0,n})_{\Q}$ with kernel $\mathcal{I}$. The
subspace $\mathcal{I}$ is generated by the Keel relations among
boundary divisors (see Equation (\ref{keelRelations})). Lastly, let
$\mathcal{F} \subseteq N^1(\M_{0,n})_{\Q}$ be the cone of $F$-nef divisor
classes. We can thus restate the conjecture as follows:

\begin{conjecture}[Fulton's conjecture, convex geometry formulation]
\label{conj:Fconvex}
For every $\alpha \in \mathcal{F}$, $\phi^{-1}(\alpha)$ intersects the first orthant $\langle \D_{J,J^c}\rangle_{\geq 0}$ nontrivially.
\end{conjecture}

By an inductive argument, Fulton's conjecture implies that the nef cone and the cone of $F$-nef divisors coincide (see \cite{morr}), or, by duality, the $F$-conjecture; it is not clear that implication holds in the other direction. Either conjecture, if true, would yield surprising consequences. It would follow that, on the level of curves, $\M_{0,n}$ behaves like a Fano variety, even though $\M_{0,n}$ is Fano only for $n \leq 5$. Moreover, by the Bridge Theorem of \cite{MR1887636}, the $F$-conjecture for $\M_{0, g+n}$ implies the analogous result for $\M_{g,n}$, yet for $g=22$ (\cite{f}) and $g \geq 24$ (\cite{MR664324} and \cite{MR910206}), $\M_{g}$ is of general type, i.e. in some sense as far as possible from being Fano. Recent work of Gibney, however, has enabled a computer-assisted proof of the $F$-conjecture for $\M_{g}$ for $g \leq 24$ (\cite{MR2551995}).

We now describe our proof, which uses techniques from both algebraic
and convex geometry. A main obstacle to proving Fulton's conjecture
is the lack of a canonical basis for $N^1(\M_{0,n})_{\Q}$. In Section
\ref{secF:FFforKeel}, we define an $\binom{n}{4}$-dimensional subspace of $N^1(\M_{0,n})_{\Q}$, called the \emph{Keel subspace}, that admits a canonical basis, meaning that every $F$-divisor class written in this basis has a natural representative as an effective sum of boundary divisors. This result holds for all $n$; see Theorem \ref{thm:FFKeel}. In the notation of Conjecture \ref{conj:Fconvex}, there is a section $\sigma$ of $\phi: V \to N^1(\M_{0,n})_{\Q}$ such that $\sigma(\mathcal{F})$ is contained in the positive orthant $\langle \D_{J, J^c} \rangle_+$. In particular, for $F$-nef classes in the Keel subspace, there is both an obvious choice of representative divisor and a recipe for combining $F$-inequalities to prove non-negativity of all boundary coefficients. Our approach for Fulton's conjecture is extend to a basis of
$N^1(\M_{0,n})_{\Q}$ so that this choice of divisor and recipe for
combining inequalities give every $F$-nef divisor in $\M_{0,n}$ as an
effective sum of boundary. For both $n=6$ and 7, this approach proves
the conjecture for a codimension one subspace of $N^1(\M_{0,n})_{\Q}$; see
Corollary \ref{cor:FFfor56} and Proposition \ref{prop:M07codim1}. We
complete the proof of Fulton's conjecture only for $n=7$ (Theorem
\ref{thm:Fnef-M07}), as the case $n=6$ has been proven elsewhere
(\cite{fa} and \cite{MR1938752}). For an $F$-nef divisor $D$ outside
of this subspace, some of the coefficients of boundary divisors can be
negative, so to finish the proof of Fulton's conjecture we exploit the
symmetry of Keel relations to find a different representative of $D$
as an effective sum of boundary divisors. We determine which inequalities establish non-negativity of this new representative via the simplex algorithm for $n=7$ (for $n=6$, these inequalities are not difficult to find by hand). We give these inequalities as intersections with explicit sums of $F$-curves. It is then straightforward to verify by hand that the new representative is an effective sum of boundary divisors. We hope that this approach will extend to $n > 7$, but for for higher $n$ it is likely impracticable to give inequalities as explicit sums of $F$-curves as is done in the Appendix.

To conclude this section, we collect notation and basic facts that are used throughout the paper. The first matter deserving comment is our choice of name for the conjecture. We follow \cite{MR1938752}, \cite{MR2483934}, and \cite{MR2551995} in using `Fulton's conjecture' to mean a modified version of Fulton's original question from \cite{km}. For standard definitions and notation from Mori theory, we refer to \cite{MR2095471}. We use without further mention that linear, rational, and numerical equivalence coincide for $\M_{0,n}$ (\cite{MR1034665}); for notational simplicity we refer mostly to numerical equivalence. Divisors are denoted by capital letters, e.g. $D \in
\mathrm{Div}_{\Q}(\M_{0,n})$, while a divisor's class in the
N\'eron-Severi space is denoted by brackets around a representative
divisor, e.g. $[D] \in N^1(\M_{0,n})_{\Q}$ (we make an exception for the hyperplane class $H \in N^1(\M_{0,n})_{\Q}$). By the `obvious representative' for a divisor class $\delta$ satisfying $\delta = [D]$, we mean the choice of representative $D$.

Central objects of study in this paper are the $F$-curves defined above. Every $F$-curve can be represented by a copy of $\M_{0,4}$ (called the \emph{spine} of the curve) with marked points $p_1, \ldots, p_4$, plus fixed pointed rational curves with $n_i + 1$ marked points, called \emph{tails}, glued to each of the marked points on the spine (see \cite{MR1887636}). If $n_i = 1$, after gluing we stabilize the corresponding tail, contracting it to the point $p_i$. The marked points on the tails of an $F$-curve give a partition of $\1tn$ into four subsets, $(\cA, \cB, \cC, \cD)$. This partition uniquely determines the numerical equivalence class of the $F$-curve (\cite{km}). We thus denote an $F$-curve $C$ as $C(\cA, \cB, \cC, \cD)$. To reduce clutter, subsets of $\1tn$ are written without brackets, and sets with more than one element are stacked vertically; for example, $C\big(1, 2, 3, \substack{4\\5}\big)$ is an $F$-curve in $\M_{0,5}$. The standard notation for boundary divisors is to label them as $\D_J$, where $J \subseteq \{1, \ldots, n\}$ records the marked points on the component of the generic element with fewer marked points. If $|J| = \frac{n}{2}$, we usually stipulate that $1 \in J$. With this labeling scheme, for each $j$ such that $2 \leq j \leq n/2$, we define the sum of boundary divisors
\begin{equation*}
B_j = \sum_{|J| = j} \D_J. 
\end{equation*}
We follow this labeling convention for boundary divisors except at the end of Section \ref{secF:FFforKeel}, where we use a different convention to facilitate calculations involving Kapranov's blow-up construction of $\M_{0,n}$.

\textbf{Acknowledgements:}
I would first like to thank my advisor Gavril Farkas for pointing me towards the nef cone of $\M_{0,n}$, for his support throughout this project, and for many explanations and suggestions. I am grateful to Frank Sottile for suggesting connections to linear programming and to G\"unter Ziegler for valuable clarifications regarding the simplex algorithm. I have also benefitted greatly from discussions with Anil Aryasomayajula, David Jensen, Kartik Venkatram, and Filippo Viviani. Finally, I would like to thank the anonymous referee for the detailed comments and suggestions for improvements.
\section{Fulton's conjecture for Keel classes}
\label{secF:FFforKeel}
The relations among boundary divisor classes in $N^{1}(\M_{0,n})_{\Q}$ are generated by the obvious equality of divisors $\D_J = \D_{J^c}$ plus the so-called Keel relations,
\begin{equation}\label{keelRelations}
\sum_{\substack{ \{i,j\} \subseteq T, \\ \{k,l\} \subseteq T^c}}[\D_T] =\sum_{\substack{ \{i,k\} \subseteq T, \\ \{j,l\} \subseteq T^c}}[\D_T]
= \sum_{\substack{ \{i,l\} \subseteq T, \\ \{j,k\} \subseteq T^c}}[\D_T],
\end{equation}
where $T \subseteq \{1, \ldots, n\}$ with $|T|, |T^c| \geq 2$ (\cite{MR1034665}). The main actors in this paper are divisors obtained by averaging the obvious representatives of Keel relations, that is, for each $n \geq 4$, we define the \emph{Keel divisor} $S_{I}$ as follows: let $I=\{i, j, k, l\} \subseteq \{1, \ldots, n\}$ and set
\begin{equation}\label{keelDivisor}
S_I = \frac{1}{3} \bigg(\sum_{\substack{ \{i,j\} \subseteq T, \\ \{k,l\} \subseteq T^c}}\D_T + \sum_{\substack{ \{i,k\} \subseteq T, \\ \{j,l\} \subseteq T^c}}\D_T
+ \sum_{\substack{ \{i,l\} \subseteq T, \\ \{j,k\} \subseteq T^c}}\D_T \bigg),
\end{equation}
where again, $T \subseteq \{1, \ldots, n\}$ such that $|T|, |T^c| \geq 2$. The class of a Keel divisor is denoted $[S_I]$, and is called a \emph{Keel class}. Note that these are defined as $\mathbb{Q}$-divisors, but on the level of numerical equivalence, we could take any of the three summands as our definition. The particular choice of $S_I$ in Equation (\ref{keelDivisor}) is important for finding a representative in a given numerical equivalence class, all of whose boundary divisor coefficients are non-negative.

\begin{lemma}
The $\binom{n}{4}$ Keel classes $[S_I]$ are linearly independent in $N^{1}(\M_{0,n})_{\Q}$.
\end{lemma}

The intuition behind this lemma is that any non-trivial linear relation among the Keel classes would give a relation in $N^1(\M_{0,n})_{\Q}$ in addition to the Keel relations, contradicting that the Keel relations generate all relations among boundary divisors. Rather than making this idea precise, we defer the proof to the discussion at the end of the present section, where it appears as Corollary \ref{KeelIndep}.

Keel classes exhibit very nice intersection properties with respect to $F$-curves. These intersections can be determined by standard calculations in $\M_{0,n}$ (see \cite{MR895568} or \cite{MR1631825}), or alternatively via intersection theory on $\P^1$. 
\begin{lemma}\label{KeelFaberIntersections}
Let $(\cA, \cB, \cC, \cD)$ be a partition of $\{1, \ldots, n\}$, and let $C$ be the corresponding $F$-curve. If $a_I = |\cA \cap I|$, and similarly for $b_I, c_I$ and $d_I$, then
\begin{displaymath}
[S_I] \cdot C = \left\{ \begin{array}{ll}
1 & \textrm{if $a_I = b_I = c_I = d_I = 1$},\\
0 & \textrm{otherwise.}
\end{array} \right.
\end{displaymath}
\end{lemma}

\begin{proof}
Let $\pi_I: \M_{0,n} \to \M_{0,4} \cong \P^1$ be the forgetful morphism that remembers only the marked points labeled by $I$. It is easily checked that, if $I=\{i,j,k,l\}$, 
\begin{equation}
S_{I} = \pi^*_I\bigg(\frac{1}{3}(\D_{ij} + \D_{ik} + \D_{il})\bigg). 
\end{equation}
Since all points of $\P^1$ are numerically equivalent, it follows that $[S_I] = \pi^*_I ([pt])$. By the push-pull formula, $[S_I] \cdot C = \pi^*_I ([pt]) \cdot C = ([pt])\cdot ((\pi_I)_*(C))$. Finally, $(\pi_I)_* ([C]) = [\M_{0,4}] =  [\P^1]$ precisely when $a_I = b_I = c_I = d_I = 1$, and is 0 otherwise.
\end{proof}

Let $n\geq 4$ and let $\cK_n$ be the $\binom{n}{4}$-dimensional subspace of $N^{1}(\M_{0,n})_{\Q}$ generated by the Keel classes. We call $\cK_n$ the \emph{Keel subspace} of $N^{1}(\M_{0,n})_{\Q}$.


\begin{lemma}\label{cJ}
The coefficient $c_J$ of $\D_J$ in $D = \sum s_I S_I = \sum c_J \D_J$ is
\begin{equation}
c_J = \frac{1}{3}\sum_{|I \cap J| = 2} s_I. \nonumber
\end{equation}
\end{lemma}
\begin{proof}
By definition, each $\D_J$ with $|I \cap J|=2$ appears with multiplicity $1/3$ in $S_I$.
\end{proof}

\begin{example}\label{M05}
To illustrate notation, we consider the well-known example of $\M_{0,5}$. The Keel divisors here are 
\begin{align}
S_{1234} &= \frac{1}{3}\big((\D_{12} + \D_{34}) + (\D_{13} + \D_{24}) + (\D_{14} + \D_{23})\big), \nonumber \\
S_{1235} &= \frac{1}{3}\big((\D_{12} + \D_{35}) + (\D_{13} + \D_{25}) + (\D_{15} + \D_{23})\big), \nonumber \\
S_{1245} &= \frac{1}{3}\big((\D_{12} + \D_{45}) + (\D_{14} + \D_{25}) + (\D_{15} + \D_{24})\big), \nonumber \\
S_{1345} &= \frac{1}{3}\big((\D_{13} + \D_{45}) + (\D_{14} + \D_{35}) + (\D_{15} + \D_{34})\big), \textrm{ and }\nonumber \\
S_{2345} &= \frac{1}{3}\big((\D_{23} + \D_{45}) + (\D_{24} + \D_{35}) + (\D_{25} + \D_{34})\big). \nonumber
\end{align}
If $D = \sum s_I S_I$, the coefficient of $\D_{12}$ in $D$ is 
\begin{align}
c_{12} &= \frac{1}{3}(s_{1234} + s_{1235} + s_{1245}) = \frac{1}{3} \sum_{|I \cap \{1,2\}| = 2} s_I. \nonumber
\end{align}
\end{example}

\begin{theorem}[Fulton's conjecture for $\cK_n$]
\label{thm:FFKeel}
If $[D] \in \cK_n$ is an $F$-nef divisor class, then $D=\sum s_I S_I$ is an effective sum of boundary divisors.
\end{theorem}
\begin{proof}
We show that for each $J \subseteq \{1, \ldots, n\}$ with $2 \leq |J| \leq n/2$, and each four-tuple ${\bf t}= (n_1, n_2, n_3, n_4) \in \mathbb{N}^4$ with $n_1 + n_2 + n_3 + n_4 = n$, there exists a collection of $F$-curves, $F^{\bf t}_J$, and a positive integer $m_{\bf t}$, such that
\begin{equation}
D \cdot \sum_{C \in F^{\bf t}_{J}} C = m_{\bf t} c_J\nonumber.
\end{equation}
To define the collection $F^{\bf t}_J$, note that the four-tuple ${\bf t}= (n_1, n_2, n_3, n_4)$ determines a type of $F$-curve, where each $n_i$ gives the cardinality of the $i^{th}$ element of a partition determining an $F$-curve.  Up to a possible reordering of partitions, define
\begin{equation}\label{eq:FJ}
  F^{\bf t}_{J} = \{C(\cA, \cB, \cC, \cD): C\textrm{ of type }{\bf t}, \, \cA \cup \cB = J\}. \nonumber
 \end{equation}
We claim that $F^{\bf t}_J$ is the set of $F$-curves $C$ of type ${\bf t}$ such that $C \cdot S_I = 0$ unless $|I \cap J| = 2$. In other words, $F^{\bf t}_J$ consists of all $F$-curves of type ${\bf t}$ whose intersection with $D$ gives inequalities involving only the $s_I$ that appear in $c_J$ as in Lemma \ref{cJ}.

To see this, suppose first that $C$ is of type ${\bf t}$, and that there exists an $S_I$ such that $C \cdot S_I = 1$ with $|I \cap J| \neq 2$. Using the notation of Lemma \ref{KeelFaberIntersections}, $a_I = b_I = c_I = d_I = 1$, so $|I \cap J| \neq 2$ implies that either there are elements of $J^{c}$ in $\cA$ or $\cB$ (for $|I \cap J| < 2$), or there are elements of $J$ in $\cC$ or $\cD$ (for $|I \cap J|> 2$), hence $C \notin F^{\bf t}_{J}$. 

Conversely, if $C$ is of type ${\bf t}$ and there are elements of $J$ in more than two of the partitions, then we can find a four-tuple $I$ with $|I \cap J| > 2$ and $C \cdot S_I = 1$. Hence $J$ must be contained in two of the partitions, $\cA$ and $\cB$, for example. If there is an element of $J^c$ in $\cA \cup \cB$, then there is an $I$ with $|I \cap J| \leq 1$, but $C \cdot S_I = 1$, hence proving the claim.

To calculate $m_{\bf t}$, select a term $s_I$ in the sum $c_J$. We count how many $F$-curves $C \in F^{\bf t}_{J}$ intersect $S_I$ with value 1. Set $\{i,j\} = I \cap J$ (since $s_I$ appears in $c_J$, we know $|I \cap J|=2$), and $\cA \cup \cB=J$. If $|\cA|=|\cB|$, there is no loss of generality in assuming $i \in \cA$ and $j \in \cB$. Then the remaining elements of the partitions $\cA$ and $\cB$ can be chosen in $\binom{|J|-2}{|\cA|-1}$ ways. If $|\cA| \neq |\cB|$, we must also consider $i \in \cB$ and $j \in \cA$, giving in total $2\binom{|J|-2}{|\cA|-1}$ ways to choose $\cA$ and $\cB$. Similarly, there are $\binom{n - |J|-2}{|\cC|-1}$ ways to choose $\cC$ and $\cD$ if $|\cC| = |\cD|$, and $2\binom{n - |J|-2}{|\cC|-1}$ ways if $|\cC| \neq |\cD|$.

Hence for each $s_I$ in $c_J$, precisely $m'=(2 - \d_{|\cA|, |\cB|})(2 - \d_{|\cC|, |\cD|})\binom{|J|-2}{|\cA|-1} \binom{n - |J|-2}{|\cC|-1}$ curves $C \in F^{\bf t}_{J}$ intersect $S_I$ with value 1. Therefore 
\begin{equation}
D \cdot \sum_{C \in F^{\bf t}_{J}} C = m' \sum_{|I \cap J| = 2} s_I = 3\, m' c_J, \nonumber
\end{equation}
so we see $m_{\bf t}=3\,m'$. Since $D$ is an $F$-nef divisor, and $m_{\bf t} \geq 0$, it follows that $c_J \geq 0$.
\end{proof}

\begin{example*}[1, continued]
There is only one type of $F$-curve in $\M_{0,5}$, so to prove that
$c_{12} \geq 0$, we consider
\begin{equation}
F^{(1:1:1:2)}_{12} = \big{\{}C\big(1 , 2, 3, \substack{4\\5}\big), C\big(1, 2, 4, \substack{3\\5}\big), C\big(1 ,2, 5, \substack{3\\4}\big) \big{\}}. \nonumber
\end{equation}
For $D = \sum s_I S_I$, the corresponding $F$-inequalities are
\begin{align}
D \cdot C(1, 2, 3, \substack{4\\5}) &= s_{1234} + s_{1235} \geq 0,  \nonumber \\
D \cdot C(1, 2, 4, \substack{3\\5}) &= s_{1234} + s_{1245} \geq 0,  \nonumber \\
D \cdot C(1, 2, 5, \substack{3\\4}) &= s_{1235} + s_{1245} \geq 0. \nonumber
\end{align}
Hence if $[D]$ is an $F$-nef divisor,
\begin{equation}
D \cdot \sum_{C \in F^{(1:1:1:2)}_{12}} C = 2\, (s_{1234} + s_{1235} + s_{1245}) = 6\, c_{12} \geq 0. \nonumber
\end{equation}
\end{example*}

\begin{corollary}
\label{cor:FFfor56}
Fulton's conjecture is true for $\M_{0,5}$ and a codimension one-subspace of $N^1(\M_{0,6})$.
\end{corollary}
\begin{proof}
The Keel subspace has dimension $\binom{n}{4}$, while $\dim N^1(\M_{0,n}) = 2^{n-1} -
\binom{n}{2} - 1$.
\end{proof}
In particular, Theorem \ref{thm:FFKeel} gives less information as $n$
increases. In Section \ref{secF:FFfor7}, however, we complete the proof of
Fulton's conjecture for $n=7$. 

To conclude this section, we relate the Keel subspace to Kapranov's construction of $\M_{0,n}$ as an iterated blow-up of $\P^{n-3}$ along linear centers (see \cite{MR1237834}, although the order of blow-ups described here is from \cite{MR1957831}). For $n\geq6$, the last two stages of blow-ups in the Kapranov construction produce $\binom{n-1}{n-5}+\binom{n-1}{n-4} = \binom{n}{4}$ exceptional divisors, so it is natural to ask how the Keel subspace relates to the subspace of $N^1(\M_{0,n})_{\Q}$ spanned by exceptional divisor classes from these last two stages of blow-ups.

Let $\mathrm{E}(j)$ be the set of exceptional divisor classes from the $j^{th}$ stage of the Kapranov construction, let $H$ be the pull-back of the hyperplane class from $\P^{n-3}$, and let $\mathrm{Exc}(k)$ be the subspace of $N^{1}(\M_{0,n})_{\Q}$ generated by $H$ and $\mathrm{E}(j)$ for $j \leq k$.

\begin{proposition}\label{KapKeel}
Let $\pi_{\mathrm{Exc}(n-6)}: N^{1}(\M_{0,n})_{\Q} \to N^{1}(\M_{0,n})_{\Q}/\mathrm{Exc}(n-6)$ be the projection map. Then $\pi_{\mathrm{Exc}(n-6)}$ restricted to the Keel subspace is an isomorphism.
\end{proposition}

\begin{corollary}\label{KeelIndep} The Keel classes are linearly independent in $N^1(\M_{0,n})_{\Q}.$
\end{corollary}
\begin{proof}
$N^{1}(\M_{0,n})_{\Q}/\mathrm{Exc}(n-6)$ is isomorphic to the subspace of $N^1(\M_{0,n})_{\Q}$ generated by $\mathrm{E}(n-5)$ and $\mathrm{E}(n-4)$. Thus $\dim(\cK_n) = \binom{n}{4}$, which equals the number of Keel classes.
\end{proof}

\begin{corollary}\label{BasisExt} The hyperplane class plus the classes of divisors from the first $n-6$ stages of the Kapranov construction extend the Keel classes to a basis of $N^1(\M_{0,n})_{\Q}$ for $n\geq 6$.
\end{corollary}

The proof of Proposition \ref{KapKeel} follows by showing a particular
intersection matrix is full rank, where all intersections are
described in the next lemma. For the remainder of this section, we use
a different indexing scheme for boundary divisors occurring as
exceptional divisors in the Kapranov construction. Namely, we write
each such divisor uniquely as $\D_J$ with $3 \leq |J| \leq n-2$ and $n
\notin J$. For example, elements of $\mathrm{E}(n-5)$ are written as
$[\D_{J}]$ where $n \notin J$ and $|J|=4$, and likewise elements of
$\mathrm{E}(n-4)$ are $[\D_{J}]$, where $n \notin J$ and $|J|=3$. We
denote elements of the dual Kapranov basis by $[\D_{J}]^{\vee}$ and
$H^\vee$. For present purposes, we need only the intersections
$[\D_J]^{\vee} \cdot [S_I]$ for $3 \leq |J| \leq 4$, but with no
additional effort, we can calculate the intersections of Keel divisors
with all 1-cycles $[\D_J]^\vee$.

\begin{lemma}\label{KapKeelInt}
Let $n \notin J$ with $3 \leq |J| \leq n-2$. Then
$[\D_J]^{\vee} \cdot [S_I]  = \mathrm{min}\{0, 2-|I \cap J|\}$. \nonumber
\end{lemma}
\begin{proof}
We first consider how 1-cycles of the dual Kapranov basis for $N_1(\M_{0,n})_{\Q}$ intersect the $\D_{ij}$ with $n \notin\{i,j\}$. Note that these are the only boundary divisors not appearing as exceptional divisors in the Kapranov construction.  We write such boundary divisors in the Kapranov basis as
\begin{equation}\label{Dij}
[\D_{ij}] = H - \sum_{\stackrel{J \supseteq \{i,j\}, n \notin J,}{3 \leq |J| \leq n-2}}[\D_J]. 
\end{equation}
It follows that the intersections $[\D_J]^{\vee} \cdot \D_{ij}$ are given by $-1$ times the number of times $[\D_J]$ appears in the right hand side of Equation (\ref{Dij}), that is,
\begin{displaymath}
[\D_J]^{\vee} \cdot \D_{ij} = \left\{ \begin{array}{ll}
0 & \textrm{if $\{i,j\} \nsubseteq J$,}\\
-1 & \textrm{if }\{i,j\} \subseteq J.
\end{array} \right.
\end{displaymath}

Now it is easy to calculate the intersection $[\D_J]^{\vee} \cdot S_I$: it is $1/3$ times the number of times $\D_J$ appears in $S_I$ minus the number of times $\D_{ij}$ appears in $S_I$, where $\{i,j\} \subseteq J$, i.e.
\begin{align}
[\D_J]^{\vee} \cdot S_I &= \frac{1}{3} \bigg(
\left\{ \begin{array}{ll}
1 & \textrm{if $|I \cap J|=2$}\\
0 & \textrm{if }|I \cap J| \neq 2
\end{array} \right\}
-\big| \big{\{} \{i,j\}: \{i,j\} \subseteq I \cap J\big{\}}\big| \bigg) \nonumber\\
&=\frac{1}{3}\bigg(
\left\{ \begin{array}{ll}
1 & \textrm{if $|I \cap J|=2$}\\
0 & \textrm{if }|I \cap J| \neq 2
\end{array} \right\}
- \binom{|I \cap J|}{2}
  \bigg).
\end{align}
Since $0 \leq |I \cap J| \leq 4$, by enumerating the five possibilities, we obtain precisely the desired formula.
\end{proof}

\begin{proof}[Proof of Proposition \ref{KapKeel}]
We can write each Keel class as a linear combination of classes in the Kapranov basis by calculating the intersection matrix given by pairing the Keel classes and the dual Kapranov basis. The projection $\pi_{\mathrm{Exc}(n-6)}$ restricted to the Keel subspace is given by setting intersections with all but the 1-cycles dual to $\mathrm{E}(n-5)$ and $\mathrm{E}(n-4)$ to zero, i.e. by the the matrix pairing Keel classes and the 1-cycles dual to $\mathrm{E}(n-5)$ and $\mathrm{E}(n-4)$. The proposition will follow by showing that the determinant of this matrix is non-zero. 

Let $M$ be the intersection matrix of Keel classes with 1-cycles dual to $\mathrm{E}(n-5)$ and $\mathrm{E}(n-4)$. $M$ can be decomposed into the following blocks:
\begin{displaymath}
M =
\left( \begin{array}{cc}
A= ([\D_J]^{\vee} \cdot [S_I])_{\substack{n \notin I\cup J,  \\ |J|=4 }}
 & B=([\D_J]^{\vee} \cdot   [S_I])_{\substack{n \in I \setminus J\\ |J|=4}},\\
C=   ( [\D_J]^{\vee} \cdot [S_I] )_{\substack{n \notin I \cup J, \\ |J|=3}}
  & D= ( [\D_J]^{\vee}  \cdot [S_I])_{\substack{n \in I \setminus J, \\ |J|=3}}
\end{array} \right),
\end{displaymath}
 so the rows of this matrix are indexed by $k$-tuples $J$, $3 \leq k \leq 4$, while the columns are indexed by four-tuples $I$.
%
We claim that $\det(M) = 2^{\binom{n-1}{4}} \, (-1)^{\binom{n-1}{3}}$. In particular, $M$ is nonsingular.

By Lemma \ref{KapKeelInt}, $D=\mathrm{diag}(-1, \ldots, -1)$, and hence 
\begin{equation}
\det(M) = \det(A - BD^{-1}C) \det(D) = (-1)^{\binom{n-1}{3}} \det(A + BC).\nonumber
\end{equation}

The rows of $BC$ are labeled by indices $J$ with $n \notin J$, $|J|=4$, while the columns are indexed by sets $I$ with $n \notin I$. With this labeling scheme,
\begin{align}
(BC)_{J,I}&=
 \sum_{\substack{n \notin J' ,\\ |J'|=3}} (  [\D_J]^{\vee} \cdot [S_{J' \cup \{n\}}]) ([\D_{J'}]^{\vee} \cdot [S_{I}]). \nonumber
\end{align}
Since $J'$ of the sum has cardinality three, and since $n \notin J$, $|I \cap J'|$ and $|(J' \cup \{n\}) \cap J|$ are at most three, by Lemma \ref{KapKeelInt} both $ [\D_J]^{\vee} \cdot  [S_{J' \cup \{n\}}]$ and $[\D_{J'}]^{\vee}  \cdot [S_{I}] $ are greater than $-2$, and their product is non-zero precisely when both equal $-1$. Therefore
\begin{align}
(BC)_{J,I}&= |\{J': |J' \cap (I \cap J)| = 3\}| \nonumber\\
&= \left\{ \begin{array}{ll}
4 & \textrm{if $I=J$},\\
1 & \textrm{if $|I \cap J| =3$},\\
0 & \textrm{if $|I \cap J| \leq 2$},
\end{array} \right. \nonumber
\end{align}
while
\begin{align}
(A)_{J,I}&= \min\{0, 2 - |I \cap J|\} \nonumber \\
&=\left\{ \begin{array}{rl}
-2 & \textrm{if $I=J$},\\
-1 & \textrm{if $|I \cap J| =3$},\\
0 & \textrm{if $|I \cap J| \leq 2$}.
\end{array} \right. \nonumber
\end{align}

Hence $A+BC=diag(2, \ldots, 2)$, proving the claim and the proposition.
\end{proof}


  \section{Fulton's conjecture for $\overline{M}_{0,7}$}
  \label{secF:FFfor7}
The dimension of the Keel subspace $\mathcal{K}_7$ is 35, while  $\dim N^1(\M_{0,7})_{\Q}=42$, so we require an additional seven divisor classes to extend to a basis. For $i,j=1, \ldots, 7$, we define $D_i = \sum_{j \neq i} \D_{ij}$.
Our basis extension candidate is then
\begin{equation}
P_i = \a D_i + \l B_2 + \m B_3,
\end{equation}
where $i=1, \ldots, 7$.
We first look for $\a, \l$, and $\m$ that will extend the proof of  Fulton's conjecture for $\cK_7$ to all of $N^{1}(\M_{0,7})_{\Q}$. As with $\M_{0,6}$, this is possible for a codimension one subspace of $N^{1}(\M_{0,7})_{\Q}$, while outside of this subspace we use averages of Keel relations to find an effective representative.
\begin{lemma}\label{basisM07}
 The Keel classes plus the $[P_i]$, $i=1, \ldots, 7$, form a basis for $N^{1}(\M_{0,7})_{\Q}$ if and only if $\a \neq 0$ and $18 \a + 63 \l - 35 \m \neq 0$.
\end{lemma}
\begin{proof}
Corollary \ref{BasisExt} states that the Keel classes plus the hyperplane class and the boundary classes $[\D_{i7}]$, $i=1, \ldots, 6$, form a basis for $N^{1}(\M_{0,7})_{\Q}$, so to ensure that the $[P_i]$ complete the Keel classes to a basis, we calculate which $\alpha$, $\l$, and $\m$ guarantee that the projection map $\pi_{\cK_{7}}: N^1(\M_{0,7})_{\Q} \to N^1(\M_{0,7})_{\Q}/\cK_7$ restricted to the span of the $[P_i]$ give an isomorphism. Concretely, we write the $[P_i]$ in terms of the basis of Corollary \ref{BasisExt}, and check which values of $\a$, $\l$, and $\m$ correspond to a non-zero determinant for the base-change matrix. The class of $D_7$ is already written in terms of this basis; for $i=1, \ldots, 6$ we propose the following  form for $[D_i]$:
\begin{equation}\label{DiAnsatz}
D_i = a \sum_{i,7\in I} S_I + b \sum_{\substack{i \in I \\ 7 \notin I}} + c \sum_{\substack{i \notin I \\ 7 \in I }} S_I + d \sum_{i,7 \notin I} S_I + e\, \D_{i7} + f \sum_{\substack{j=1, \ldots, 6 \\ j \neq i}}\D_{j7} + h \, H.
\end{equation}
For concreteness, we determine $[D_1]$, the other $[D_i]$ being completely analogous. Intersections of all $F$-curves of $\M_{0,7}$ with the left and right sides of the proposed equality (\ref{DiAnsatz}) for $i=1$ are as follows, where, by a statement such as `1 on $(2)$,' we mean that the marked point 1 is on a tail containing two marked points:
\begin{align}
(*) \left\{
\begin{array}{rll}
2=& a + 3\,b, &\text{$C$ type $(1:1:1:4)$, with 1 on spine, 7 on tail},\\
2= &4\,a + e + f  +h, &\text{$C$ type $(1:1:1:4)$, with 1, 7 on spine}, \\
0= &a + 3\,c + 2\,f +h, &\text{$C$ type $(1:1:1:4)$, with 1 on tail, 7 on spine}, \\
0= &b+c + 2\,d, &\text{$C$ type $(1:1:1:4)$, with 1, 7 on tail}, \\
1=&3\, a + 3\,b - f, &\text{$C$ type $(1:1:2:3)$, with 1 on spine, 7 on $(2)$},\\
1=&2\, a + 4\,b, &\text{$C$ type $(1:1:2:3)$, with 1 on spine, 7 on $(3)$},\\
1=& 6\, a + e + h, &\text{$C$ type $(1:1:2:3)$, with 1, 7 on spine},\\
-1=& 3\,b + 3\,c - e,   &\text{$C$ type $(1:1:2:3)$, with 1, 7 on $(2)$},\\
0=& 2\,b + 2\,c + 2\,d,   &\text{$C$ type $(1:1:2:3)$, with 1, 7 on $(3)$},\\
-1=& 3\,a + 3\, c + f +h,   &\text{$C$ type $(1:1:2:3)$, with 1 on $(2)$, 7 on spine},\\
 0=& 2\,a + 4\, c + f +h,   &\text{$C$ type $(1:1:2:3)$, with 1 on $(3)$, 7 on spine},\\
0=& 4\,a + 4\, b - f, &\text{$C$ type $(1:2:2:2)$, with 1 on spine, 7 on tail},\\
-1=& 4\,a + 4\,c + h, &\text{$C$ type $(1:2:2:2)$, with 1 on tail, 7 on spine},\\
-1=& 4\,b + 4\,c - e, &\text{$C$ type $(1:2:2:2)$, with 1, 7 on same tail},\\
-1=& 2\,a + 2\,b + 2\,c
+ 2\,d - f, &\text{$C$ type $(1:2:2:2)$, with 1, 7 on distinct tails}.
\nonumber
\end{array} \right.
\end{align}
Gauss-Jordan elimination gives the (overdetermined but consistent) solution
\begin{equation*}
 [D_1] = -\frac{5}{2} \sum_{1,7 \in I} + \frac{3}{2}\Big(\sum_{\substack{1 \in I \\ 7 \notin I}} [S_I] - \sum_{\substack{1 \notin I \\7 \in I}} [S_I]\Big) + [\D_{17}] - 4\sum_{j=2}^6 [\D_{j7}] + 15 [H].
\end{equation*}

To write $[B_2]$ and $[B_3]$ in terms of Keel classes, the $[\D_{i7}]$, $1 \leq i \leq 6$, and $H$, we may consider a simpler expression than that for $D_i$, and we need only differentiate $F$-curves according to where the marked point $7$ is. Alternatively we may use the same expression and simply add two more columns to the right of the augmented matrix with the intersections of $[B_2]$ and $[B_3]$ with $C$, $C'$, and $C''$, respectively $(1:1:1:4)$, $(1:1:2:3)$, and $(1:2:2:2)$ $F$-curves:
\begin{align*}
&[B_2]\cdot C = 3, \; \, \, \, \quad[B_2]\cdot C' = 0, \quad [B_2]\cdot C'' = -3, \\
&[B_3]\cdot C = - 1, \quad [B_3]\cdot C' = 1,  \quad [B_3] \cdot C'' = 3.
\end{align*}
The result is
\begin{align*}
[B_2] &= 3 \sum_{7 \notin I} S_I -6 \sum_{7 \in I} S_I - 9 \sum_{1 \leq j \leq 6 } [\D_{j7}] + 45 H,\nonumber \\
[B_3] &= -\frac{3}{2} \sum_{7 \notin I} S_I +\frac{7}{2} \sum_{7 \in I} S_I + 5 \sum_{1 \leq j \leq 6 } [\D_{j7}] - 25 H.\nonumber
\end{align*}

Setting $x=-4\a - 9\l + 5\m$, the matrix given by the coordinates of $[\D_{j7}]$, $1 \leq j \leq 6$ and $H$ for the $[P_i]$ is
\begin{equation*}
P = \left(
\begin{array}{ccccc}
x+5\a & x & \ldots & x & -5(x + \a) \\
x & x+5\a &  & x & -5(x + \a)\\
\vdots &  & \ddots& & \vdots\\
x & x&  & x + 5\a & -5(x + \a) \\
x+ 5\a & x+5\a & \ldots & x+ 5\a & -5(x + 4\a)
\end{array} \right).
\end{equation*}
Row and column operations, followed by a Laplace expansion, show that the matrix has determinant $-5 (5\a)^6 (-5(x + 4\a) + 6(\frac{3}{5}(x+5\a)))$, which, after substituting for $x$, yields the desired result,
$\det(P) = -(5\a)^6(18 \a + 63 \l - 35 \m)$.
\end{proof}

Sufficient conditions for the determinant of Lemma \ref{basisM07} to vanish are easy to find. In particular, $\sum S_I = \frac{10}{3} \, B_2 + 6 \,B_3$, while $\sum P_i = (2 \, \a + 7\, \l) \, B_2 + 7 \, \m \, B_3$. These two sums are proportional if and only if $18 \a + 63 \l - 35 \m = 0$. Moreover, the choice $\a = 0$ obviously implies that the matrix is singular.

For $\a\,(18 \a + 63 \l - 35 \m) \neq 0$, every divisor class $[D]$ can be written uniquely as
\begin{equation}\label{DM07}
[D] = \sum s_I [S_I] + \sum_{i=1}^7 p_i [P_i].
\end{equation}
Taking the obvious representative, the coefficients of $\D_{12}$ and $\D_{123}$ in $D$ are
\begin{align*}
c_{12} & = \frac{1}{3} \sum_{I \supseteq \{1,2\}} s_I  +(\a + \l)(p_1 + p_2) + \l (p_3 + \ldots + p_7),\\
c_{123} & = \frac{1}{3} \sum_{|I \cap \{1,2,3\}| = 2} s_I + \m (p_1 + \ldots + p_7).
\end{align*}


The cone of $F$-nef divisors is determined by the following three sets of inequalities, corresponding to partitions $(\cA, \cB, \cC, \cD)$ of types $(1:1:1:4), (1:1:2:3),$ and $(1:2:2:2)$, respectively:

\begin{align}
\label{F4ineq07} \sum_{r=1}^4 s_{I_r}+ \, &(2\a+ 3\l - \m)\sum_{i \in \cA \cup \cB \cup \cC} p_i + (3\l - \m) \sum_{j \in \cD}p_j  \geq 0,\\
\label{F3ineq07} \sum_{r=1}^6 s_{I'_r} + \, & (\a + \m)\sum_{i \in \cA \cup \cB} p_i + (-\a + \m) \sum_{j \in \cC} p_j + \m \sum_{k \in \cD} p_k  \geq 0,\\
\label{F2ineq07}\sum_{r=1}^8 s_{I''_r} + \, &3(- \l + \m) \sum_{i \in \cA} p_{i} + (-\a + 3(- \l + \m)) \sum_{j \in \cB \cup \cC \cup \cD} p_j    \geq 0,
\end{align}
where all index sets are distinct within a given inequality, and, as in the notation of Lemma \ref{KeelFaberIntersections}, satisfy $a_{I_r} = b_{I_r} = c_{I_r} = d_{I_r} = 1$ for all $r$ (respectively $a_{I'_r} = \ldots = d_{I'_r} =1$, $a_{I''_r} = \ldots = d_{I''_r} = 1$). For example, the inequality of type (\ref{F3ineq07}) for the curve $C(1, 2, \substack{3\\4}, \substack{5\\6\\7})$ is
\begin{align*}
s_{1235}& + s_{1236} + s_{1237} + s_{1245} + s_{1246} + s_{1247} \\
+(\a&  + \m)(p_1 + p_2) + (-\a + \m)(p_3 + p_4) + \m(p_5+p_6+p_7) \geq 0. 
\end{align*}

\begin{theorem}[ (Fulton's conjecture for $\M_{0,7}$]
\label{thm:Fnef-M07}
Let $[D]$ be an $F$-nef divisor class in $\M_{0,7}$. Then $D$ is numerically equivalent to an effective sum of boundary divisors.
\end{theorem}
\begin{proof}
By symmetry, the theorem is proven if we show that the coefficients $c_{12}$ and $c_{123}$ of the obvious representative $D$ from Equation (\ref{DM07}) are non-negative. The proof of Theorem \ref{thm:FFKeel} gives inequalities for $c_{12}$ from $(1:1:1:4)$ and $(1:1:2:3)$ $F$-curves. We will combine these inequalities to conclude $c_{12} \geq 0$ without placing further restrictions on the parameters $\a$, $\l$, and $\m$.

We consider first 
\begin{align}
F^{(1:1:1:4)}_{12} = \big{\{}&C(1, 2, 3, \substack{4\\5\\6\\7}), C(1, 2, 4, \substack{3\\5\\6\\7}),C(1, 2, 5, \substack{3\\4\\6\\7}), C(1, 2, 6, \substack{3\\4\\5\\7}), C(1, 2, 7, \substack{3\\4\\5\\6}) \big{\}}.\nonumber
\end{align}
By the proof of Theorem \ref{thm:FFKeel} and Equation (\ref{F4ineq07}), we obtain
\begin{equation}
D \cdot \sum_{C \in F^{(1:1:1:4)}_{12}}C
=2 \sum_{I \supseteq \{1,2\}}s_I + (10\a + 15\l - 5\m)(p_1 + p_2) +(2\a + 15\l - 5\m) (p_3 + \ldots + p_7). \nonumber
\end{equation}
Substituting $6 c_{12}$ yields
\begin{align*}
6 c_{12} + (4\a + 9 \l - 5\m)(p_1 + p_2) 
+ (2\a + 9\l - 5\m)(p_3 + \ldots + p_7) \geq 0.
\end{align*}

Next, for
\begin{align}
F^{(1:1:2:3)}_{12} = \{ &C(1, 2, \substack{3\\4}, \substack{5\\6\\7}), C(1, 2, \substack{3\\5},\substack{4\\6\\7}), C(1, 2, \substack{3\\6}, \substack{4\\5\\7}), C(1, 2, \substack{3\\7}, \substack{4\\5\\6}), C(1, 2, \substack{4\\5}, \substack{3\\6\\7}), \nonumber \\
& C(1,2,\substack{4\\6}, \substack{3\\5\\7}), C(1,2,\substack{4\\7},\substack{3\\5\\6}),
C(1,2,\substack{5\\6}, \substack{3\\4\\7}), C(1,2,\substack{5\\7}, \substack{3\\4\\6}), C(1,2,\substack{6\\7}, \substack{3\\4\\5})\}, \nonumber
\end{align}
we obtain as above
\begin{align*}
18c_{12} +  (-8\a -18 \l+ 10\m)(p_1+p_2) 
+(-4\a - 18\l +10\m) (p_3 + \ldots + p_7)\geq 0. \nonumber
\end{align*}
Combining, we see that for any allowed values of $\a$, $\l$, and $\m$,
\begin{equation}
D \cdot \bigg(\frac{1}{15}\sum_{C \in F^{(1:1:1:4)}_{12}} C + \frac{1}{30}\sum_{C' \in F^{(1:1:2:3)}_{12}}C'\bigg) = c_{12} \geq 0. \nonumber
\end{equation}

Next we consider
\begin{align*}
F^{(1:1:2:3)}_{123} = \big{\{}&C( 4, 1, \substack{2\\3}, \substack{5\\6\\7}), 
C( 5, 1, \substack{2\\3}, \substack{4\\6\\7}),
C( 6, 1, \substack{2\\3}, \substack{4\\5\\7}),
C( 7, 1, \substack{2\\3}, \substack{4\\5\\6}), 
C( 4, 2, \substack{1\\3}, \substack{5\\6\\7}), 
C( 5, 2, \substack{1\\3}, \substack{4\\6\\7}),\\
& 
C( 6, 2, \substack{1\\3}, \substack{4\\5\\7}),
C( 7, 2, \substack{1\\3}, \substack{4\\5\\6}), 
C( 4, 3, \substack{1\\2}, \substack{5\\6\\7}),   
C( 5, 3, \substack{1\\2}, \substack{4\\6\\7}),
C( 6, 3, \substack{1\\2}, \substack{4\\5\\7}),   
C( 7, 3, \substack{1\\2}, \substack{4\\5\\6})\big{\}}, \nonumber 
\end{align*}
%
where we have reordered the partitions so that $\cB \cup \cC = \{1,2,3\}$.
 By the proof of Theorem \ref{thm:FFKeel} and Equation (\ref{F3ineq07}),
\begin{align*}
D \cdot \sum_{C \in F^{(1:1:2:3)}_{123}} C
        =4 &\sum_{|I \cap \{1,2,3\}| = 2} s_I + (-4\a + 12\m)(p_1 + p_2 + p_3)
+ (3\a + 12\m) (p_4 + \ldots + p_7). \nonumber 
\end{align*}
Substituting for $c_{123}$ gives
$12 c_{123} -4 \a (p_1 + p_2 + p_3) + 3\a ( p_4+ \ldots + p_7) \geq 0.$

Finally, for
\begin{align}
F^{(1:2:2:2)}_{123} = \big{\{}&C(1, \substack{2\\3}, \substack{4\\5}, \substack{6\\7}), C(1, \substack{2\\3}, \substack{4\\6}, \substack{5\\7}), C(1, \substack{2\\3}, \substack{4\\7}, \substack{5\\6}), C(2, \substack{1\\3}, \substack{4\\5}, \substack{6\\7}), C(2, \substack{1\\3}, \substack{4\\6}, \substack{5\\7}), \nonumber \\
&  C(2, \substack{1\\3}, \substack{4\\7}, \substack{5\\6}), C(3, \substack{1\\2}, \substack{4\\5}, \substack{6\\7}), C(3, \substack{1\\2}, \substack{4\\6}, \substack{5\\7}), C(3, \substack{1\\2}, \substack{4\\7}, \substack{5\\6}) \big{\}},\nonumber
\end{align}
we obtain
$12 c_{123} + (-6\a - 27 \l + 15 \m)(p_1 + p_2 + p_3) 
 + (-9\a - 27 \l + 15 \m) (p_4 + \ldots + p_7) \geq 0$.

Suppose next that $\a > 0$. The first inequality for $c_{123}$ implies $c_{123} \geq 0$ provided that $\frac{4}{3}(p_1 + p_2 + p_3) \geq p_4 + \ldots + p_7$. Supposing then that $p_4 + \ldots + p_7 > \frac{4}{3}(p_1 + p_2 + p_3)$ and picking $\a, \l,$ and $\m$ such that $9 \a + 27 \l - 15 \m \geq 0$, the second inequality for $c_{123}$ implies
\begin{equation}
12 c_{123} \geq (18 \a + 63 \l - 35 \m)(p_1 + p_2 + p_3), \nonumber
\end{equation}
so the conditions on $\a$, $\l$, and $\m$ that extend the proof from the Keel subspace ensure that the basis candidate under consideration does not yield a basis. Nevertheless, for $\a \neq 0$ and $18 \a + 63\l - 35 \m = 0$, the rank of the intersection matrix $P$ of Lemma \ref{basisM07} is six, thus proving the following proposition.
\begin{proposition}
\label{prop:M07codim1}
Let $[D]$ be an $F$-nef divisor class in the codimension one subpace generated by the $[S_I]$ and $[P_i]$, where $\a > 0$, $9 \a + 27 \l - 15 \m \geq 0$, and $18 \a + 63\l - 35 \m = 0$. Then the obvious representative is an effective sum of boundary divisors.
\end{proposition}
 
 Outside of this codimension one subspace, the obvious representative of an $F$-nef divisor class written as in Equation (\ref{DM07}) can have boundary terms with negative coefficients. To find an effective representative, we use Keel relations to obtain a new representative, and combine $F$-inequalities to prove that all boundary coefficients of this new representative are non-negative. Within the codimension one subspace of Proposition \ref{prop:M07codim1}, we followed the proof of Fulton's conjecture for the Keel subspace (Theorem \ref{thm:FFKeel}) to find suitable combinations of $F$-inequalities. Outside of this subspace, we use instead the simplex algorithm (see \cite{MR0201189}, or \cite{MR1311028} for a succinct, geometrical account).
 
A principal use of the simplex algorithm is to minimize (or maximize) a linear functional subject to linear constraints. The algorithm terminates if the linear functional is expressed so that any move within the feasible region results in either no change, or an increase (or decrease) to the functional. By convexity, local minima (or maxima) give the global minimum (or maximum) of the functional.

In the present case, we take the coefficients of boundary divisors of the obvious representative $D$ as the linear functionals, and use the simplex algorithm to obtain lower bounds subject to the $F$-inequalities. These bounds can then be verified by hand, since the coefficient in question is given at the last stage of the algorithm as a non-negative sum of the inequalities defining the feasible region. We use the implementation of the simplex algorithm in $lp\_solve$ (\cite{lpsolve}) because it enables us to read off the linear functional from the final stage.

We may assume that $c_{123}$ is the most negative of the boundary coefficients, and, by scaling, that $c_{123} = -1$. We fix a basis by setting $\a=3, \l=5$, and $\m=9$.
Consider first the average of Keel relations for $-[\D_{123}]$ as in \cite{MR1938752}:
\begin{align}\label{D123avgM07}
-[\D_{123}] = &\frac{1}{3} \sum_{a,b \in \{1,2,3\}}[\D_{ab}] 
+ \frac{1}{6}\sum_{x,y \in \{4,5,6,7\}}[\D_{xy}]
- \frac{1}{6}\sum_{\substack{a \in \{1,2,3\},\\ x,y \in \{4,5,6,7\}}}[\D_{axy}]\\
&+ \frac{1}{2}\sum_{x,y,z \in \{4,5,6,7\}}[\D_{xyz}] 
- \frac{1}{6}\sum_{\substack{a \in \{1,2,3\}, \\ x \in \{4,5,6,7\}}}[\D_{ax}]  
  \nonumber
\end{align}
(we stipulate without further mention that indices in expressions as above are distinct and chosen to avoid double counting). Substituting the average of the obvious representative of $-\D_{123}$ from (\ref{D123avgM07})  into $D$ yields an effective representative provided that the following four sets of inequalities are satisfied: $(i)$ $c_{abx} \geq 0$, $(ii)$ $c_{axy} \geq 1/6$, $(iii)$ $c_{xyz} \geq -1/2$, and $(iv)$ $c_{ax} \geq 1/6$ for all $a,b \in \{1,2,3\}$ and $x,y,z \in \{4,5,6,7\}$. To finish the proof, we consider each of these four collections of inequalities, and show that either a given inequality is satisfied, or that we can replace $-\D_{123}$ by a different average so that the resulting divisor is an effective sum of boundary. In most cases, we will actually find a sharper bound than is required.

For $(i)$, we prove $c_{abx} \geq 3$ for all $a, b \in \{1,2,3\}$ and all $x \in \{4,5,6,7\}$. By symmetry, it suffices to prove $c_{124} \geq 3$. The last stage of the simplex algorithm gives $c_{124}$ as the intersection of $D$ with the following sum of $F$-curves:
\begin{gather*}
\frac{1}{10}(C(1,2,5,\substack{3\\4\\6\\7}) + C(1,\substack{2\\3}, \substack{4\\7}, \substack{5\\6}) + C\big(2,\substack{1\\3}, \substack{4\\5}, \substack{6\\7}) + C(2, \substack{1\\3}, \substack{4\\6}, \substack{5\\7}) + C(2, \substack{1\\3}, \substack{4\\7}, \substack{5\\6}) + C(5, \substack{1\\2}, \substack{3\\4}, \substack{6\\7})\big) 
\\
+\frac{4}{15} C(2,5,\substack{1\\3}, \substack{4\\6\\7})
+ \frac{1}{30}\big(C(2,4, \substack{6\\7}, \substack{1\\3\\5}) + C(2,5, \substack{3\\4}, \substack{1\\6\\7}) + C(2,6,\substack{1\\3}, \substack{4\\5\\7}) + C(2,7, \substack{1\\3}, \substack{4\\5\\6})\big) 
+\frac{4}{5} C(3,4, \substack{1\\2}, \substack{5\\6\\7})
\\%
+\frac{1}{6}\big(C(2,6, \substack{5\\7}, \substack{1\\3\\4}) + C(2,7,\substack{5\\6}, \substack{1\\3\\4}) + C(1, \substack{2\\3}, \substack{4\\5}, \substack{6\\7})\big) %
+\frac{2}{15}\big(C(1,6, \substack{3\\4}, \substack{2\\5\\7}) + C(1,7, \substack{3\\4}, \substack{2\\5\\6}) + C(1, \substack{2\\3}, \substack{4\\6}, \substack{5\\7})\big) 
  \\
+\frac{1}{30}\big(C(1,4,6, \substack{2\\3\\5\\7}) + C(1,5, \substack{2\\3}, \substack{4\\6\\7}) + C(1,6, \substack{2\\3}, \substack{4\\5\\7}) + C(1,7, \substack{2\\3}, \substack{4\\5\\6})  +C(1,7,\substack{5\\6}, \substack{2\\3\\4}) \big)
+\frac{1}{15}\big(C(1,4,5, \substack{2\\3\\6\\7})\big)
\\
+\frac{1}{5}\big(C(1,2,6, \substack{3\\4\\5\\7}) + C(1,2,7, \substack{3\\4\\5\\6}) + C(3,6, \substack{1\\2}, \substack{4\\5\\7}) + C(3,7,\substack{1\\2}, \substack{4\\5\\6}) + C(1,3,5, \substack{2\\4\\6\\7})\big)  
 + \frac{1}{15}C(2,4,\substack{1\\3}, \substack{5\\6\\7})\big) .
\end{gather*}
Intersecting the above with $D$ and substituting $-3 \, c_{123} = 3$ gives $c_{124} \geq 3$.

The remaining three sets of inequalities fail, so we must find different average expressions for $-[\D_{123}]$ in each case. In particular, we must now specify which of the Keel relations involving $\D_{123}$ to include in the average. Each four-tuple $(ijkl)$ determines three Keel relations, and $\D_{123}$ appears in these relations precisely when (up to reordering) $i,j \in \{1,2,3\}$ and $k,l \in \{4,5,6,7\}$. We will write the two relations involving $\D_{123}$ as $(ij)(kl) = (ik)(jl)$ and $(ij)(kl) = (il)(jk)$. To prove non-negativity of boundary coefficients for the new representative of $D$ requires several sums of $F$-curves as with $c_{abx} \geq 3$, so we give only the resulting bounds below, and record the sums of $F$-curves giving these bounds in an appendix.

{\bf Case $(ii)$: $-1 \leq c_{axy} \leq 1/6$.}
By symmetry, we consider only $c_{145}$. The Keel relation $(23)(45) = (24)(35)$ gives
\begin{align}
-[\D_{123}] = & [\D_{236}] + [\D_{237}] +[\D_{145}] +  [\D_{456}] + [\D_{457}] + [\D_{23}] + [\D_{45}] \nonumber \\
&-([\D_{124}]+ [\D_{246}] + [\D_{247}] + [\D_{135}] + [\D_{356}] + [\D_{357}]) - ([\D_{24}] + [\D_{35}]). \nonumber
\end{align}
Substituting the obvious representative of $-\D_{123}$ produces an effective representative of $D$ provided no $c_{ijk}$ besides $c_{123}$ and $c_{145}$ is negative, $c_{246} \geq 1$, and $c_{24} \geq 1$, since the other inequalities follow from symmetry.

We obtain $c_{246} \geq 17/3$ by substituting $-44/7 \, c_{123} = 44/7$ into the sum of the inequality $-26/7 \, c_{145} \geq -26/42$ and the intersection of $D$ with the sum of $F$-curves $C^{(ii)}_{246}$ of the appendix. The bound $c_{24} \geq 1$ follows by substituting $-8/7 \, c_{123} = 8/7$ into the sum of $-6/7 \, c_{145} \geq -1/7$ and the intersection of $D$ with the sum of $F$-curves $C^{(ii)}_{24}$ of the appendix.

It remains to show that no other $c_{ijk}$ can be negative. By symmetry, the following inequalities (more than) suffice after after substituting $c_{123} = -1$:
\begin{enumerate}
\item  $c_{146} \geq 73/6$, which follows from  $-65/7 \, c_{145} + D \cdot C^{(ii)}_{146}\geq -65/42$;
 \item $c_{167} \geq 83/6,$ which follows from  $-7\, c_{145} + D \cdot C^{(ii)}_{167} \geq -7/6$;
 \item $c_{245} \geq 37/6,$ which follows from  $-29/7 \, c_{145} + D \cdot C^{(ii)}_{245}  \geq -29/42$;
 \item $c_{267} \geq 59/6,$ which follows from  $-25/7 \, c_{145} + D \cdot C^{(ii)}_{267} \geq -25/42$;
 \item $c_{456} \geq 9,$ which follows from  $-36/7 \, c_{145} + D \cdot C^{(ii)}_{456} \geq -6/7$; and
\item $c_{467} \geq 13/2,$ which follows from $-45/7 \, c_{145} +  D \cdot C^{(ii)}_{467}\geq -45/42$.
\end{enumerate}

{\bf Case $(iii)$: $-1 \leq c_{xyz} \leq 0$.}
For Equation (\ref{D123avgM07}) to give an effective representative, we require only the upper bound $c_{456} \leq -1/2$, but considering $c_{456} \leq 0$ simplifies case $(iv)$. The divisors $\D_{123}$ and $\D_{456}$ both appear in 18 Keel relations corresponding to the four-tuples (1245), (1246), (1256), (1345), (1346), (1356), (2345), (2346), and (2356), so we average over these relations, that is, we average over (12)(45) = (14)(25), (12)(45) = (15)(24), $\ldots \,$, (23)(56) = (25)(36), (23)(56) = (26)(35) to obtain:
%
\begin{align*}
-[\D_{123}] = &[\D_{456}] -\frac{2}{9}\sum_{\substack{a \in \{1,2,3\}, \\ x \in \{4,5,6\}}}\big( [\D_{ax}]  + [\D_{ax7}] \big) 
-\frac{1}{9}\big( \sum_{\substack{a,b \in \{1,2,3\}, \\ x \in \{4,5,6\}}} [\D_{abx}] + \sum_{\substack{a \in \{1,2,3\}, \\ x,y \in \{4,5,6\}}}[\D_{axy}] \big)\\
&+\frac{1}{3}\big( \sum_{a,b \in \{1,2,3\}} [\D_{ab}] +\sum_{x,y \in \{4,5,6\}} [\D_{xy}] 
        +\sum_{a,b \in \{1,2,3\}} [\D_{ab7}] + \sum_{x,y \in \{4,5,6\}} [\D_{xy7}] \big). \nonumber 
\end{align*}
%
The following bounds show that substituting the obvious representatives from the above average for $-\D_{123}$ gives an effective representative of $D$:
\begin{enumerate}
\item $c_{14} \geq 1$, which follows from  $-1/3 \, c_{456} + D \cdot C^{(iii)}_{14} \geq 0$;
\item $c_{145} \geq 13/3$, which follows from $-4/3 \, c_{456} + D \cdot C^{(iii)}_{145} \geq 0$;
\item $c_{147} \geq 53/45$, which follows from $-7/45 \, c_{456} + D \cdot C^{(iii)}_{147} \geq 0$; and
\item $c_{457} \geq 2$, which follows from $-c_{456} + D \cdot C^{(iii)}_{457} \geq 0$.
\end{enumerate}

{\bf Case $(iv)$: $c_{ax} \leq 1/6$.}
We may now assume that $c_{axy} \geq 1/6$ and $c_{xyz} \geq 0$ for all $a \in \{1,2,3\}$ and $x,y,z \in \{4,5,6,7\}$. Suppose $c_{14} \leq 1/6$. We average over the 18 Keel relations involving $\D_{123}$ but not $\D_{14}$ corresponding to the four-tuples (1256), (1257), (1267), (1356), (1357), (1367), (2356), (2357), and (2367), that is, over the relations $(12)(56) = (15)(26)$, $(12)(56) = (16)(25)$, \ldots, $(23)(67) = (26)(37)$, $(23)(67) = (27)(36)$:
\begin{align*}
-[\D_{123}] = &[\D_{567}] -\frac{2}{9}\sum_{\substack{a \in \{1,2,3\}, \\ x \in \{5,6,7\}}}\big( [\D_{ax}]  + [\D_{a4x}] \big)
-\frac{1}{9}\big( \sum_{\substack{a,b \in \{1,2,3\}, \\ x \in \{5,6,7\}}} [\D_{abx}] + \sum_{\substack{a \in \{1,2,3\}, \\ x,y \in \{5,6,7\}}}[\D_{axy}] \big)
 \nonumber \\
        &+\frac{1}{3}\big(\sum_{a,b \in \{1,2,3\}} [\D_{ab}] +  \sum_{x,y \in \{5,6,7\}} [\D_{xy}] 
        +\sum_{a,b \in \{1,2,3\}} [\D_{ab4}] + \sum_{x,y \in \{5,6,7\}} [\D_{4xy}] \big). \nonumber \\
\end{align*}
The following bounds show that substituting the obvious representative for $-\D_{123}$ from the above average makes $D$ an effective sum of boundary:
\begin{enumerate}
\item $c_{15} \geq 7/6$, which follows from $-c_{14} +  D \cdot C^{(iv)}_{15}\geq -1/6$;
\item $c_{145} \geq 257/198$, which follows from $-7/33 \, c_{14} + D \cdot C^{(iv)}_{145} \geq -7/198$; and
\item $c_{245} \geq 41/36$, which follows from $-7/6 \, c_{14} + D \cdot C^{(iv)}_{245} \geq -7/36$.
\end{enumerate}
This completes the proof of Theorem \ref{thm:Fnef-M07}.
 \end{proof}
\section{Appendix: $F$-curves used in the proof of theorem \ref{thm:Fnef-M07}}\label{appendix}
We record the sums of $F$-curves that give the inequalities on boundary coefficients used in the proof of Theorem \ref{thm:Fnef-M07}.
\newline \newline
\noindent{\bf Case $(ii)$}: $-1 \leq c_{145} \leq 1/6$.
\newline
\noindent For $c_{246} \geq 17/3$:
\begin{align*}
\begin{array}{rc}
C^{(ii)}_{246} =& \frac{7}{30} \big(C(1,5,6, \substack{2\\3\\4\\7}) + C(2,\substack{1\\3},\substack{4\\7}, \substack{5\\6}) + C(3, \substack{1\\2}, \substack{4\\6}, \substack{5\\7})\big) +\frac{43}{70}C(1,3, \substack{4\\5}, \substack{2\\6\\7})
+\frac{22}{105}C(5,6, \substack{1\\4}, \substack{2\\3\\7})
\\
&+\frac{1}{3}\big((C(3,4,7, \substack{1\\2\\5\\6}) + C(2,7, \substack{4\\6}, \substack{1\\3\\5}) + C(4,\substack{1\\5}, \substack{2\\6},\substack{3\\7})\big)
+\frac{22}{105}C(1,2, \substack{4\\5}, \substack{3\\6\\7})
+\frac{51}{210}C(5,7, \substack{1\\4}, \substack{2\\3\\6}) 
 \\
&+\frac{11}{35}\big(C(1,3,6, \substack{2\\4\\5\\7}) + C(2,3,4, \substack{1\\5\\6\\7})\big)
+\frac{11}{30}\big((C(1, \substack{2\\3}, \substack{4\\5}, \substack{6\\7}) + C(2, \substack{1\\3}, \substack{4\\6}, \substack{5\\7}) + C(3, \substack{1\\2}, \substack{4\\7}, \substack{5\\6}) \big)
\\
&+\frac{29}{70}C(2,3,5, \substack{1\\4\\6\\7}) +\frac{2}{105}C(2,3,6, \substack{1\\4\\5\\7}) 
+\frac{19}{42}\big((C(3,4,5, \substack{1\\2\\6\\7}) + C(2,4, \substack{6\\7}, \substack{1\\3\\5})\big)
 \\
&+\frac{46}{105}C(1,4, \substack{2\\3}, \substack{5\\6\\7})
+\frac{4}{7}C(1,5, \substack{2\\3}, \substack{4\\6\\7})
+\frac{58}{105}C(1,6, \substack{4\\5}, \substack{2\\3\\7})
+\frac{9}{10}C(1,7, \substack{2\\3}, \substack{4\\5\\6}) \\
&+\frac{1}{30}\big(C(2,4, \substack{1\\3}, \substack{5\\6\\7}) 
 + C(3,5, \substack{1\\2}, \substack{4\\6\\7})\big)
 +\frac{11}{70}C(3,7, \substack{1\\2}, \substack{4\\5\\6})
+\frac{17}{210}C(3, \substack{1\\2}, \substack{4\\5}, \substack{6\\7})  
\\
&+\frac{24}{35}C(2,5, \substack{1\\3}, \substack{4\\6\\7})
+\frac{33}{70}C(2,6, \substack{1\\3}, \substack{4\\5\\7})
+\frac{13}{105}\big((C(2,7, \substack{1\\3}, \substack{4\\5\\6}) + C(3,6, \substack{1\\2}, \substack{4\\5\\7})\big) \\
&
 +\frac{1}{10}\big((C(4,5, \substack{2\\6}, \substack{1\\3\\7}) + C(4, \substack{1\\5}, \substack{2\\7}, \substack{3\\6})\big)
 +\frac{10}{21}C(4,6, \substack{1\\5}, \substack{2\\3\\7})
 +\frac{23}{210}C(4,7, \substack{1\\5}, \substack{2\\3\\6}).
\end{array}
\end{align*}

\noindent For $c_{24} \geq 1$:
\begin{align*}
\begin{array}{rc}
C^{(ii)}_{24} =&
\frac{4}{35} C(1,2,5, \substack{3\\4\\6\\7})
+\frac{5}{210} \big(C(1,3,5, \substack{2\\4\\6\\7}) + C(1,6, \substack{2\\3}, \substack{4\\5\\7}) + C(3,5, \substack{1\\4}, \substack{2\\6\\7})\big)  
+\frac{16}{105} C(2,4,5, \substack{1\\3\\6\\7}) 
\\
&+\frac{11}{210} C(3,4,5, \substack{1\\2\\6\\7})
+\frac{13}{210} \big(C(1,5, \substack{2\\3}, \substack{4\\6\\7}) + C(3,6, \substack{1\\2}, \substack{4\\5\\7}) + C(3,7, \substack{1\\2}, \substack{4\\5\\6})\big) 
+\frac{11}{105} C(2,4, \substack{1\\3}, \substack{5\\6\\7})
\\
&+\frac{37}{210} C(1,4, \substack{2\\3}, \substack{5\\6\\7}) 
+\frac{1}{7} C(1,6, \substack{4\\5}, \substack{2\\3\\7}) 
+\frac{1}{6} C(1,7, \substack{2\\3}, \substack{4\\5\\6})
+\frac{13}{105} C(2,3, \substack{4\\5}, \substack{1\\6\\7})
+\frac{1}{10} C(2,4, \substack{5\\6}, \substack{1\\3\\7})
 \\
&+\frac{31}{210} C(2,4, \substack{1\\5}, \substack{3\\6\\7}) 
+\frac{1}{15} \big(C(2,4, \substack{3\\6}, \substack{1\\5\\7}) + C(2, 4, \substack{3\\7}, \substack{1\\5\\6}) + C(3,5, \substack{6\\7}, \substack{1\\2\\4})\big)  \\
&+\frac{1}{210} C(2,4, \substack{5\\7}, \substack{1\\3\\6}) 
+\frac{1}{105} C(2,4, \substack{6\\7}, \substack{1\\3\\5})
+\frac{2}{21} \big(C(2,6, \substack{5\\7}, \substack{1\\3\\4})  + C(2, 7, \substack{1\\3}, \substack{4\\5\\6})\big) \\
&+\frac{1}{14} \big(C(4,6, \substack{1\\5}, \substack{2\\3}) + C(4,7, \substack{1\\5}, \substack{2\\3\\6}) + C(5,6, \substack{1\\4}, \substack{2\\3\\7}) + C(5,7, \substack{1\\4}, \substack{2\\3\\6})\big)\\
&+\frac{11}{70} C(1, \substack{2\\3}, \substack{4\\5}, \substack{6\\7}) 
+\frac{1}{30} \big(C(3, \substack{1\\2}, \substack{4\\6}, \substack{5\\7}) + C(3, \substack{1\\2}, \substack{4\\7}, \substack{5\\6})\big)
+\frac{3}{14} C(2,3,4, \substack{1\\5\\6\\7}).
\end{array}
\end{align*}

\noindent For $c_{146} \geq 73/6$:
\begin{align*}
\begin{array}{rc}
C^{(ii)}_{146} =& \frac{22}{15} C(1,2,5, \substack{3\\4\\6\\7}) 
 + \frac{1}{14} C(1,3,5, \substack{2\\4\\6\\7}) 
  + \frac{7}{30} C(1,3,6, \substack{2\\4\\5\\7}) 
   + \frac{206}{105} C(2,3,4, \substack{1\\5\\6\\7}) 
  + \frac{19}{35} C(2,6, \substack{1\\3}, \substack{4\\5\\7}) 
\\
& + \frac{1}{3} \big(C(3,4,7, \substack{1\\2\\5\\6}) + C(4,5,7, \substack{1\\2\\3\\6}) + C(2,6, \substack{5\\7}, \substack{1\\3\\4})\big) 
 + \frac{242}{105} C(1,4, \substack{2\\3}, \substack{5\\6\\7}) 
+ \frac{16}{105} C(1,6,\substack{2\\3}, \substack{4\\5\\7}) 
    \\
&  + \frac{12}{7} C(1,6,\substack{4\\5}, \substack{2\\3\\7})
       + \frac{1}{3} \big(C(3,5,\substack{1\\4}, \substack{2\\6\\7}) + C(2, \substack{1\\3}, \substack{4\\6}, \substack{5\\7}) + C(4, \substack{1\\5}, \substack{2\\6}, \substack{3\\7}) \big)
+ \frac{4}{5} \big(C(2, \substack{1\\3}, \substack{4\\7}, \substack{5\\6})\big)  
\\
& + \frac{1}{10} C(2,3, \substack{4\\6}, \substack{1\\5\\7})
         + \frac{46}{105} C(2,4, \substack{1\\5}, \substack{3\\6\\7}) 
                   + \frac{12}{35} \big(C(2,7, \substack{1\\3}, \substack{4\\5\\6}) + C(3,7, \substack{1\\2}, \substack{4\\5\\6}) \big) 
+  \frac{4}{5}\big(C(3, \substack{1\\2}, \substack{4\\7}, \substack{5\\6}) \big)
\\
& + \frac{11}{105} C(3,4, \substack{1\\5}, \substack{2\\6\\7})
           + \frac{223}{210} C(3,5, \substack{1\\2}, \substack{4\\6\\7}) 
             + \frac{16}{35} \big(C(4,6, \substack{1\\5}, \substack{2\\3\\7}) + C(5,7, \substack{1\\4}, \substack{2\\3\\6}) \big)
+ \frac{62}{105} C(4,7, \substack{1\\5}, \substack{2\\3\\6})
\\
 &  + \frac{44}{35} \big(C(5,6, \substack{1\\4}, \substack{2\\3\\7}) + C(1, \substack{2\\3}, \substack{4\\5}, \substack{6\\7}) \big) 
                    + \frac{12}{5} C(1,7,\substack{2\\3}, \substack{4\\5\\6})
                    + \frac{13}{42} C(3,6,\substack{1\\2}, \substack{4\\5\\7})\\
&+ \frac{2}{3} C(3, \substack{1\\2}, \substack{4\\6}, \substack{5\\7})
               + \frac{33}{35} C(5, \substack{1\\4}, \substack{2\\3}, \substack{6\\7})
                + \frac{11}{15} C(2,3,5, \substack{1\\4\\6\\7}) 
                 + \frac{103}{210} C(2,3, \substack{4\\5}, \substack{1\\6\\7}).
               \end{array}
\end{align*}

\noindent For $c_{167} \geq 83/6$:
\begin{align*}
\begin{array}{rc}
C^{(ii)}_{167} =& \frac{1}{3} \big(C(1,2,7,\substack{3\\4\\5\\6})+ C(1,4,7,\substack{2\\3\\5\\6}) + C(1,5,6, \substack{2\\3\\4\\7}) + C(2,3,6, \substack{1\\4\\5\\7})\big)
+\frac{2}{15}C(1,5,\substack{2\\3}, \substack{4\\6\\7})
\\
& +\frac{37}{30}C(1,3,5, \substack{2\\4\\6\\7}) 
+ \frac{1}{3}\big(C(2,6,7, \substack{1\\3\\4\\5})  
+ C(3,4, \substack{5\\6}, \substack{1\\2\\7}) + C(3,6, \substack{4\\7}, \substack{1\\2\\5})\big) 
+\frac{9}{10}C(3,7,\substack{1\\2}, \substack{4\\5\\6}) 
\\
&+\frac{29}{15}C(2,3,4, \substack{1\\5\\6\\7}) 
+\frac{19}{30}C(3,4,5, \substack{1\\2\\6\\7}) 
+\frac{19}{15} \big(C(1,4, \substack{2\\3}, \substack{5\\6\\7}) + C(2,5,\substack{1\\3}, \substack{4\\6\\7})\big)
+\frac{7}{3}C(1, \substack{2\\3}, \substack{4\\5}, \substack{6\\7})
\\
&+\frac{7}{30} \big(C(1,3,\substack{4\\5}, \substack{2\\6\\7}) + C(2,6,\substack{1\\3}, \substack{4\\5\\7}) + C(2,7,\substack{1\\3}, \substack{4\\5\\6}) + C(3,6,\substack{1\\2}, \substack{4\\5\\7}\big)
+\frac{13}{30}C(3,\substack{1\\2}, \substack{4\\7}, \substack{5\\6})
\\
&+\frac{5}{3}C(1,6,\substack{2\\3}, \substack{4\\5\\7})
+\frac{2}{3} \big(C(1,7,\substack{2\\3}, \substack{4\\5\\6}) + C(1,7,\substack{4\\5}, \substack{2\\3\\6})\big)
+\frac{1}{6}C(2,4, \substack{1\\5}, \substack{3\\6\\7}) 
\\
&+\frac{1}{30} \big(C(2,5,\substack{1\\4}, \substack{3\\6\\7}) + C(3,5,\substack{1\\4}, \substack{2\\6\\7})\big)
+\frac{12}{15}C(2,5, \substack{3\\4}, \substack{1\\6\\7}) 
+\frac{1}{5}C(3,5,\substack{1\\2}, \substack{4\\6\\7}) 
\\
&+\frac{5}{6} \big(C(4,6,\substack{1\\5}, \substack{2\\3\\7}) + C(5,7,\substack{1\\4}, \substack{2\\3\\6})\big)
+\frac{1}{2} \big(C(4,7, \substack{1\\5}, \substack{2\\3\\6}) + C(5,6,\substack{1\\4}, \substack{2\\3\\7})\big)
\\
&+\frac{2}{5}C(2, \substack{1\\3}, \substack{4\\5}, \substack{6\\7}) 
+\frac{23}{30} \big(C(2,\substack{1\\3}, \substack{4\\6}, \substack{5\\7}) + C(2, \substack{1\\3}, \substack{4\\7}, \substack{5\\6}) + C(3, \substack{1\\2}, \substack{4\\6}, \substack{5\\7})\big).
\end{array}
\end{align*}
\noindent For $c_{245} \geq 37/6$:
\begin{align*}
\begin{array}{rc}
C^{(ii)}_{245} =& \frac{17}{14} C(1,2,4, \substack{3\\5\\6\\7})
+\frac{59}{70} C(2,3,5, \substack{1\\4\\6\\7}) 
+\frac{29}{210} C(1,4,\substack{2\\3}, \substack{5\\6\\7}) 
+\frac{11}{21} C(1,5, \substack{2\\3}, \substack{4\\6\\7}) 
+\frac{13}{35} C(2,7,\substack{1\\3}, \substack{4\\5\\6})
\\
& +\frac{6}{7} C(1,6, \substack{4\\5}, \substack{2\\3\\7}) 
 +\frac{1}{105} \big(C(2,4,\substack{5\\7}, \substack{1\\3\\6}) + C(2,7, \substack{4\\6}, \substack{1\\3\\5}) + C(4,5, \substack{2\\6}, \substack{1\\3\\7}) \big) 
+\frac{26}{105} C(2, \substack{1\\3}, \substack{4\\5}, \substack{6\\7}) 
\\
&+\frac{18}{35} C(2,5, \substack{1\\4}, \substack{3\\6\\7}) 
+\frac{8}{21} C(2,6, \substack{1\\3}, \substack{4\\5\\7}) 
 +\frac{34}{105} \big(C(2,7, \substack{5\\6}, \substack{1\\3\\4}) + C(5,6, \substack{2\\4}, \substack{1\\3\\7}) \big) 
+\frac{1}{7} C(1,6, \substack{2\\3}, \substack{4\\5\\7})\big)
\\
&+\frac{7}{10} C(3,4, \substack{1\\2}, \substack{5\\6\\7}) 
+\frac{3}{70} C(3,4, \substack{1\\5}, \substack{2\\6\\7}) 
+\frac{16}{35} C(3,4, \substack{2\\5}, \substack{1\\6\\7}) 
+\frac{5}{14} C(3,5, \substack{1\\4}, \substack{2\\6\\7}) 
+ \frac{1}{7}C(1,7, \substack{2\\3}, \substack{4\\5\\6})
\\
&+\frac{11}{42} \big(C(4,6, \substack{1\\5}, \substack{2\\3\\7}) + C(4,7, \substack{1\\5}, \substack{2\\3\\6}) + C(5,6, \substack{1\\4}, \substack{2\\3\\7}) + C(5,7, \substack{1\\4}, \substack{2\\3\\6}) \big) 
+ \frac{1}{42}C(2, \substack{1\\3}, \substack{4\\6}, \substack{5\\7})
\\
&+\frac{6}{35} \big(C(3,6, \substack{1\\2}, \substack{4\\5\\7}) + C(3,7, \substack{1\\2}, \substack{4\\5\\6})\big)
+\frac{1}{3} C(4,5, \substack{2\\7}, \substack{1\\3\\6}) 
+\frac{76}{105} C(1, \substack{2\\3}, \substack{4\\5}, \substack{6\\7})
\\
&+\frac{1}{30} C(2, \substack{1\\3}, \substack{4\\7}, \substack{5\\6}) 
+\frac{8}{35} C(3, \substack{1\\2}, \substack{4\\5}, \substack{6\\7}) 
+\frac{2}{5} \big(C(3, \substack{1\\2}, \substack{4\\6}, \substack{5\\7}) + C(3, \substack{1\\2}, \substack{4\\7}, \substack{5\\6}) \big).
\end{array}
\end{align*}

\noindent For $c_{267} \geq 59/6$:
\begin{align*}
\begin{array}{rc}
C^{(ii)}_{267} = 
&  \frac{8}{21} \big(C(4, 6, \substack{1\\5}, \substack{2\\3\\7})+C( 4, 7, \substack{1\\5}, \substack{2\\3\\6})+C( 5, 6, \substack{1\\4}, \substack{2\\3\\7})+ C( 5, 7, \substack{1\\4}, \substack{2\\3\\6}) + C( 2, \substack{1\\3}, \substack{4\\5}, \substack{6\\7}) \big) 
\\
&+\frac{104}{105} C(1, 2, 5, \substack{3\\4\\6\\7}) 
  + \frac{1}{3} \big(C(2, 3, 6, \substack{1\\4\\5\\7}) + C( 2, 3, 7, \substack{1\\5\\6\\7}) + C( 3, 6, 7, \substack{1\\2\\4\\5}) \big) 
+ \frac{8}{105} C(3, 4, 5, \substack{1\\2\\6\\7}) 
\\
&   + \frac{47}{35} C(1, 4, \substack{2\\3}, \substack{5\\6\\7}) 
    + \frac{3}{2} C(1, 6, \substack{2\\3}, \substack{4\\5\\7}) 
     + \frac{31}{42} C(1, 7, \substack{2\\3}, \substack{4\\5\\6}) 
      + \frac{16}{21} C(1, 7, \substack{4\\5}, \substack{2\\3\\6}) 
 + \frac{22}{105} C(2, 4, \substack{1\\3}, \substack{5\\6\\7}) 
\\
&       + \frac{1}{10} C(2, 4, \substack{6\\7}, \substack{1\\3\\5}) 
                 + \frac{16}{105} \big(C(2, 6, \substack{1\\3}, \substack{4\\5\\7})+C( 2, 7, \substack{1\\3}, \substack{4\\5\\6})+C( 5, \substack{1\\4}, \substack{2\\3}, \substack{6\\7}) \big) 
+ \frac{107}{210} C(2, 5, \substack{6\\7}, \substack{1\\3\\4}) 
\\
 &  + \frac{1}{30} C(3, 4, \substack{1\\2}, \substack{5\\6\\7}) 
   + \frac{43}{42} C(3, 5, \substack{1\\2}, \substack{4\\6\\7}) 
   + \frac{1}{5} C(3, 5, \substack{1\\4}, \substack{2\\6\\7}) 
    + \frac{34}{105} C(3, \substack{1\\2}, \substack{4\\5}, \substack{6\\7}) 
\\
   & + \frac{5}{42} \big(C(3, 6, \substack{1\\2}, \substack{4\\5\\7}) + C( 3, 7, \substack{1\\2}, \substack{4\\5\\6}) \big) 
     + \frac{1}{2} \big(C(3, \substack{1\\2}, \substack{4\\6}, \substack{5\\7}) + C( 3, \substack{1\\2}, \substack{4\\7}, \substack{5\\6}) \big) 
\\
& + \frac{7}{15} \big(C(1, \substack{2\\3}, \substack{4\\5}, \substack{6\\7}) + C( 2, \substack{1\\3}, \substack{4\\6}, \substack{5\\7}) + C( 2, \substack{1\\3}, \substack{4\\7}, \substack{5\\6}) \big) 
 + \frac{25}{21} C(2, 3, 4, \substack{1\\5\\6\\7}) .
\end{array}
\end{align*}
\noindent For $c_{456} \geq 9$:
\begin{align*}
\begin{array}{rc}
C^{(ii)}_{456} = & \frac{1}{30} \big(C(2,4, \substack{1\\3}, \substack{5\\6\\7}) + C(2,5, \substack{1\\4}, \substack{3\\6\\7}) + C(3,4, \substack{1\\5}, \substack{2\\6\\7}) + C(4,5, \substack{2\\7}, \substack{1\\3\\6}) + C(4,5, \substack{3\\7}, \substack{1\\2\\6}) \big) 
\\
&+ \frac{29}{70} \big(C(1,2,5, \substack{3\\4\\6\\7}) + C(5,7, \substack{1\\4}, \substack{2\\3\\6})\big) 
  + \frac{11}{210} C(1,4,7, \substack{2\\3\\5\\6}) 
   + \frac{137}{105} C(1,6,7, \substack{2\\3\\4\\5}) 
+ \frac{31}{105} C(1,6, \substack{2\\3}, \substack{4\\5\\7}) 
\\
 &  + \frac{121}{105} C(2,3,5, \substack{1\\4\\6\\7}) 
   + \frac{5}{42} C(2,6,7, \substack{1\\3\\4\\5}) 
    + \frac{82}{105} C(1,4,\substack{2\\3}, \substack{5\\6\\7}) 
     + \frac{19}{14} C(1,5, \substack{2\\3}, \substack{4\\6\\7}) 
+ \frac{19}{42} C(2,6, \substack{1\\3}, \substack{4\\5\\7}) 
\\
& + \frac{3}{70} C(1,7, \substack{2\\3}, \substack{4\\5\\6}) 
 + \frac{47}{105} C(2,4, \substack{1\\5}, \substack{3\\6\\7}) 
  + \frac{3}{14} \big(C(2,4, \substack{5\\6}, \substack{1\\3\\7}) + C(5,7, \substack{2\\4}, \substack{1\\3\\6}) \big) 
+ \frac{4}{7} C(3,6, \substack{1\\2}, \substack{4\\5\\7}) 
\\
& + \frac{2}{7} C(2,7, \substack{1\\3}, \substack{4\\5\\6})
  + \frac{1}{3} \big(C(3,4, \substack{5\\6}, \substack{1\\2\\7}) + C(3,5, \substack{4\\7}, \substack{1\\2\\6}) \big) 
   + \frac{4}{35} C(3,5, \substack{2\\4}, \substack{1\\6\\7}) 
\\
& + \frac{17}{42} C(3,7, \substack{1\\2}, \substack{4\\5\\6}) 
  + \frac{23}{70} \big(C(4,6, \substack{1\\5}, \substack{2\\3\\7}) + C(5,6, \substack{1\\4}, \substack{2\\3\\7}) + C(3, \substack{1\\2}, \substack{4\\5}, \substack{6\\7}) \big) 
\\
& + \frac{38}{105} C(4,7, \substack{1\\5}, \substack{2\\3\\6}) 
 + \frac{41}{21} C(1, \substack{2\\3}, \substack{4\\5}, \substack{6\\7} ) 
  + \frac{1}{2} \big(C(2, \substack{1\\3}, \substack{4\\6}, \substack{5\\7}) + C(3, \substack{1\\2}, \substack{4\\6}, \substack{5\\7}) \big) 
\\
& + \frac{2}{7} C(2, \substack{1\\3}, \substack{4\\7}, \substack{5\\6}) 
 + \frac{1}{6} C(3, \substack{1\\2},\substack{4\\7}, \substack{5\\6}) 
  + \frac{19}{35} C(4, \substack{1\\5}, \substack{2\\3}, \substack{6\\7})
 + \frac{37}{30} C(2,3,4, \substack{1\\5\\6\\7}) .
\end{array}
\end{align*}
\noindent For $c_{467} \geq 13/2$:
\begin{align*}
\begin{array}{rc}
C^{(ii)}_{467} =& \frac{199}{420} C(1,2,4, \substack{3\\5\\6\\7}) 
 + \frac{541}{420} C(1,2,5, \substack{3\\4\\6\\7}) 
  + \frac{10}{7} C(1,6,7, \substack{2\\3\\4\\5}) 
   + \frac{27}{35} C(2,3,4, \substack{1\\5\\6\\7}) 
  + \frac{141}{140} C(4, \substack{1\\5}, \substack{2\\3}, \substack{6\\7} )
\\
& + \frac{7}{60} \big(C(4,5,7, \substack{1\\2\\3\\6}) + C(3,4, \substack{1\\5}, \substack{2\\6\\7}) + C(5,6, \substack{3\\7}, \substack{1\\2\\4}) \big) 
    + \frac{23}{84} C(1,3, \substack{4\\5}, \substack{2\\6\\7})   
   + \frac{4}{35} C(5, \substack{1\\4}, \substack{2\\3}, \substack{6\\7} ) 
\\
& + \frac{2}{7} \big(C(1,2, \substack{4\\5}, \substack{3\\6\\7}) + C(4,6, \substack{1\\5}, \substack{2\\3\\7}) + C(4,7, \substack{1\\5}, \substack{2\\3\\6}) \big) 
 + \frac{1}{30} C(2,4, \substack{1\\3}, \substack{5\\6\\7}) 
 + \frac{17}{60} C(5, \substack{1\\4}, \substack{2\\7}, \substack{3\\6})   
\\
& + \frac{1}{84} C(2,5, \substack{3\\4}, \substack{1\\6\\7}) 
  + \frac{13}{28} C(2,6, \substack{1\\3}, \substack{4\\5\\7}) 
   + \frac{209}{420} C(2,7, \substack{1\\3}, \substack{4\\5\\6}) 
   + \frac{131}{210} C(3,4, \substack{1\\2}, \substack{5\\6\\7}) 
  + \frac{1}{3} C(3,6, \substack{4\\7}, \substack{1\\2\\5})
\\
& + \frac{13}{60} \big(C(2,7, \substack{4\\6}, \substack{1\\3\\5}) + C(4,5, \substack{2\\7}, \substack{1\\3\\7}) \big) 
 + \frac{13}{10} C(3,5, \substack{1\\2}, \substack{4\\6\\7}) 
    + \frac{26}{105} C(3,6, \substack{1\\2}, \substack{4\\5\\7})
 \\
& + \frac{139}{420} C(3,7, \substack{1\\2}, \substack{4\\5\\6}) 
      + \frac{1}{4} \big(C(4,5, \substack{2\\6}, \substack{1\\3\\7})  + C(2, \substack{1\\3}, \substack{4\\7}, \substack{5\\6}) \big) 
      + \frac{267}{140} C(1, \substack{2\\3}, \substack{4\\5}, \substack{6\\7} ) 
\\
& + \frac{53}{210} \big(C(5,6,\substack{1\\4}, \substack{2\\3\\7}) + C(5,7, \substack{1\\4}, \substack{2\\3\\6}) \big) 
    + \frac{23}{60} C(3, \substack{1\\2}, \substack{4\\6}, \substack{5\\7} ) 
        + \frac{2}{15} C(3, \substack{1\\2}, \substack{4\\7}, \substack{5\\6} ) 
\\
& + \frac{1}{14} \big(C(1,6, \substack{4\\5}, \substack{2\\3\\7}) + C(1,7, \substack{4\\5}, \substack{2\\3\\6}) \big)    
+ \frac{31}{140} C(2,4, \substack{1\\5}, \substack{3\\6\\7}) 
       + \frac{1}{12} C(4,5, \substack{3\\6}, \substack{1\\2\\7}) 
.
 \end{array}
\end{align*}
\noindent {\bf Case $(iii)$: $-1 \leq c_{456} \leq 0$.} \newline
\noindent For $c_{14} \geq 1$:
\begin{align*}
\begin{array}{rc}
C^{(iii)}_{14} = &
\frac{1}{9} \big(C(1,4,5, \substack{2\\3\\6\\7}) + C(1,6, \substack{4\\7}, \substack{2\\3\\5}) \big) 
+\frac{1}{15} \big(C(1,4, \substack{5\\6}, \substack{2\\3\\7}) +C(2, \substack{1\\3}, \substack{4\\5}, \substack{6\\7}) + C(2, \substack{1\\3}, \substack{4\\6}, \substack{5\\7}) \big)
\\
&+\frac{4}{45} \big(C(1,4,6, \substack{2\\3\\5\\7}) + C(1,5, \substack{2\\3}, \substack{4\\6\\7}) \big) 
+\frac{1}{18} \big(C(1,3,4, \substack{2\\5\\6\\7}) + C(2,5,7, \substack{1\\3\\4\\6}) + C(2,6,7, \substack{1\\3\\4\\5})\big)
\\
& +\frac{1}{18}\big( C(3,5,\substack{4\\6}, \substack{1\\2\\7}) + C(4,7, \substack{5\\6}, \substack{1\\2\\3}) + C(1, \substack{2\\3}, \substack{4\\6}, \substack{5\\7}) \big)
+\frac{1}{45} \big(C(1,4, \substack{5\\7}, \substack{2\\3\\6}) + C(3,6, \substack{4\\5}, \substack{1\\2\\7}) \big)
\\
&+\frac{1}{30} C(3,6,7, \substack{1\\2\\4\\5}) 
+\frac{7}{90} \big(C(1,4, \substack{2\\7}, \substack{3\\5\\6}) + C(1, \substack{2\\3}, \substack{4\\5}, \substack{6\\7}) + C(3, \substack{1\\2}, \substack{4\\7}, \substack{5\\6}) \big) 
+\frac{1}{5} C(1,2,4, \substack{3\\5\\6\\7})  
\\
&+\frac{2}{45} \big(C(1,4, \substack{2\\3}, \substack{5\\6\\7}) + C(2,5, \substack{1\\3}, \substack{4\\6\\7}) + C(2,6, \substack{1\\3}, \substack{4\\5\\7}) \big) +\frac{13}{90} C(3,4, \substack{1\\2}, \substack{5\\6\\7}) 
+\frac{2}{15} C(1,4,7, \substack{2\\3\\5\\6})
\\
&+\frac{1}{90} C(3, \substack{1\\2}, \substack{4\\6}, \substack{5\\7})
+ \frac{2}{45}\big(C(3,5, \substack{1\\2}, \substack{4\\6\\7}) + C(3,6, \substack{1\\2}, \substack{4\\5\\7}) + C(3, \substack{1\\2}, \substack{4\\5}, \substack{6\\7}) \big).
\end{array}
 \end{align*}
\noindent For $c_{145} \geq 13/3$:
\begin{align*}
\begin{array}{rc}
C^{(iii)}_{145} = &\frac{1}{3} \big(C(1,2,4, \substack{3\\5\\6\\7}) + C(1,2,5, \substack{3\\4\\6\\7}) + C(2,5, \substack{1\\3}, \substack{4\\6\\7}) \big) 
+\frac{2}{45} \big(C(2,4,6, \substack{1\\3\\5\\7}) + C(4,7, \substack{3\\5}, \substack{1\\2\\6}) \big) 
\\
&+\frac{11}{45} C(3,5,6, \substack{1\\2\\4\\7}) 
+\frac{13}{90} \big(C(1,5, \substack{4\\6}, \substack{2\\3\\7}) + C(2,6, \substack{4\\5}, \substack{1\\2\\7}) + C(1, \substack{2\\3}, \substack{4\\7}, \substack{5\\6}) \big)
+\frac{17}{45} C(3,5, \substack{1\\2}, \substack{4\\6\\7})
\\
&+\frac{5}{9} C(1,6,\substack{2\\3}, \substack{4\\5\\7})
+\frac{1}{9} C(1,6, \substack{4\\5}, \substack{2\\3\\7}) 
+\frac{47}{90} C(1,7,\substack{4\\5}, \substack{2\\3\\6}) 
+\frac{8}{45} \big(C(2,4, \substack{5\\6}, \substack{1\\3\\7}) + C(3, \substack{1\\2}, \substack{4\\6}, \substack{5\\7}) \big)
\\
&+\frac{1}{45} \big(C(2,4, \substack{1\\3}, \substack{5\\6\\7}) + C(4, \substack{1\\3}, \substack{2\\7}, \substack{5\\6}) + C(7, \substack{1\\5}, \substack{2\\3}, \substack{4\\6}) \big) 
+\frac{1}{30} C(1,4,7, \substack{2\\3\\5\\6}) 
+\frac{7}{15} C(3,4, \substack{1\\2}, \substack{5\\6\\7}) 
\\
&+\frac{4}{45} \big(C(2,4, \substack{6\\7}, \substack{1\\3\\5}) + C(3,5, \substack{6\\7}, \substack{1\\2\\4}) + C(4,7, \substack{1\\5}, \substack{2\\3\\6}) \big)
+\frac{1}{10} C(5,7, \substack{3\\4}, \substack{1\\2\\6}) 
\\
&+\frac{7}{90} \big(C(2,6, \substack{1\\3}, \substack{4\\5\\7}) + C(4, \substack{1\\7}, \substack{2\\3}, \substack{5\\6}) + C(5, \substack{1\\7}, \substack{2\\3},\substack{4
\\6}) \big) 
+\frac{3}{10} C(3,4, \substack{5\\6}, \substack{1\\2\\7}) 
\\
&+\frac{11}{90} C(3, \substack{1\\2}, \substack{4\\7}, \substack{5\\6})
+\frac{1}{10} C(4,5,7, \substack{1\\2\\3\\6}) 
+\frac{1}{5} \big(C(2,6, \substack{4\\7}, \substack{1\\3\\5}) + C(3,7, \substack{1\\2}, \substack{4\\5\\6}) \big)
\\
&+\frac{2}{9} \big(C(2,7, \substack{1\\3}, \substack{4\\5\\6}) + C(3,6, \substack{1\\2}, \substack{4\\5\\7}) + C(1, \substack{2\\3},\substack{4\\5}, \substack{6\\7}) + C(2, \substack{1\\3}, \substack{4\\6}, \substack{5\\7}) \big).
\end{array}
 \end{align*}

\noindent For $c_{147} \geq 53/45$:
\begin{align*}
\begin{array}{rc}
C^{(iii)}_{147} =& \frac{31}{450} C(1,2,4, \substack{3\\5\\6\\7}) 
+\frac{4}{135} C(1,3,4, \substack{2\\5\\6\\7}) 
+\frac{2}{9} C(1,5,6, \substack{2\\3\\4\\7}) 
+\frac{221}{1350} C(2,4,7, \substack{1\\3\\5\\6}) 
\\
&+\frac{28}{675} \big(C(4,5,7, \substack{1\\2\\3\\6}) + C(2,6, \substack{4\\5}, \substack{1\\3\\7}) \big) 
+\frac{7}{135} C(4,6,7, \substack{1\\2\\3\\5})
+\frac{281}{1350} C(1,2,\substack{4\\7}, \substack{3\\5\\6}) 
\\
&+\frac{59}{1350} \big(C(2,6, \substack{1\\3}, \substack{4\\5\\7}) + C(3,6, \substack{1\\2}, \substack{4\\5\\7}) \big) 
+\frac{23}{1350} C(1,7, \substack{2\\3}, \substack{4\\5\\6}) 
+\frac{139}{675} C(3,4, \substack{1\\2}, \substack{5\\6\\7}) 
\\
&+\frac{1}{90} \big(C(1,5, \substack{4\\7}, \substack{2\\3\\6}) + C(1,6, \substack{4\\7}, \substack{2\\3\\5}) \big) 
+\frac{26}{675} \big(C(2,5, \substack{1\\3}, \substack{4\\6\\7}) + C(2,7, \substack{3\\4}, \substack{1\\5\\6}) \big) 
\\
&+\frac{47}{1350} C(3,7, \substack{1\\4}, \substack{2\\5\\6}) 
+\frac{7}{150} C(3,7, \substack{4\\6}, \substack{1\\2\\5}) 
+\frac{2}{45} \big(C(4,5, \substack{1\\7}, \substack{2\\3\\6}) + C(4,6, \substack{1\\7}, \substack{2\\3\\5}) \big) 
\\
&+\frac{7}{1350} \big(C(4,7, \substack{2\\5}, \substack{1\\3\\6}) + C(5,7,\substack{3\\4}, \substack{1\\2\\6}) \big) 
+\frac{7}{675} C(5,7, \substack{3\\6}, \substack{1\\2\\4}) 
+\frac{451}{1350} C(1, \substack{2\\3}, \substack{4\\7}, \substack{5\\6}) 
\\
&+\frac{53}{1350} \big(C(5,7, \substack{1\\4}, \substack{2\\3\\6}) + C(6,7, \substack{1\\4}, \substack{2\\3\\5}) \big) 
+\frac{91}{1350} C(3, \substack{1\\2}, \substack{4\\5}, \substack{6\\7})
+\frac{7}{135} C(7, \substack{1\\4}, \substack{2\\6}, \substack{3\\5}) 
\\
&+\frac{7}{270} \big(C(3,4, \substack{1\\7}, \substack{2\\5\\6}) + C(2, \substack{1\\3}, \substack{4\\5}, \substack{6\\7}) + C(3, \substack{1\\2}, \substack{4\\6}, \substack{5\\7}) \big) 
+\frac{7}{450} C(3,4, \substack{5\\6}, \substack{1\\2\\7}) 
\\
&
+\frac{23}{270} C(3,5, \substack{1\\2}, \substack{4\\6\\7}) 
+\frac{82}{675} C(3,7, \substack{5\\6}, \substack{1\\2\\4})
+\frac{1}{1350} C(2,4, \substack{3\\7}, \substack{1\\5\\6}) 
+\frac{49}{675} C(2, \substack{1\\3}, \substack{4\\6}, \substack{5\\7}).
\end{array}
 \end{align*}
\noindent For $c_{457} \geq 2$:
\begin{align*}
\begin{array}{rc}
C^{(iii)}_{457} = & \frac{1}{10} \big( C(1, \substack{2\\3}, \substack{4\\6}, \substack{5\\7}) + C(1, \substack{2\\3}, \substack{4\\7}, \substack{5\\6}) + C(2, \substack{1\\3}, \substack{4\\6}, \substack{5\\7})
                          + C(2, \substack{1\\3}, \substack{4\\7}, \substack{5\\6}) + C(3, \substack{1\\2}, \substack{4\\6}, \substack{5\\7}) \big) 
\\
&+ \frac{1}{15} \big(  C(2,4, \substack{5\\7}, \substack{1\\3\\6}) + C(2,5, \substack{1\\3}, \substack{4\\6\\7})+ C(3,4, \substack{1\\2}, \substack{5\\6\\7}) + C(3,4, \substack{5\\6}, \substack{1\\2\\7}) \big)
+ \frac{1}{30} C(3, \substack{1\\2}, \substack{4\\7}, \substack{5\\6})
\\
&+ \frac{1}{15}\big(C(3,5, \substack{1\\2}, \substack{4\\6\\7}) + C(3, \substack{1\\2}, \substack{4\\5}, \substack{6\\7}) + C(6, \substack{1\\2}, \substack{3\\7}, \substack{4\\5}) \big) 
 + \frac{2}{15} \big( C(1,6,7, \substack{2\\3\\4\\5}) + C(2,6,7, \substack{1\\3\\4\\5}) \big) 
\\
 &+ \frac{4}{15} \big( C(3,4,7, \substack{1\\2\\5\\6}) + C(1,4, \substack{6\\7}, \substack{2\\3\\5}) + C(2, \substack{1\\3}, \substack{4\\5}, \substack{6\\7}) \big)
 + \frac{1}{5} \big( C(2,4,5, \substack{1\\3\\6\\7}) + C(3,6, \substack{1\\2}, \substack{4\\5\\7}) \big)  
\\
   &+\frac{1}{15} \big( C(1,4,6, \substack{2\\3\\7}) + C(2,5,6, \substack{1\\3\\4\\7}) + C(2,5,7, \substack{1\\3\\4\\6})
+ C(1,4, \substack{2\\3}, \substack{5\\6\\7}) \big)
\\
&+ \frac{1}{15}\big(  C(1,5, \substack{4\\6}, \substack{2\\3\\7}) 
 +C(1,7,\substack{4\\5}, \substack{2\\3\\6}) + C(2,4, \substack{1\\3}, \substack{5\\6\\7})  +
 C(2,4, \substack{5\\6}, \substack{1\\3\\7})\big) 
\\
& + \frac{1}{3} \big( C(1,5, \substack{2\\3}, \substack{4\\6\\7}) + C(3,5, \substack{4\\7}, \substack{1\\2\\6}) + C(6,7, \substack{4\\5}, \substack{1\\2\\3})\big).
\end{array}
 \end{align*}
\noindent {\bf Case $(iv)$: $c_{14} \leq 1/6$.} \newline
\noindent For $c_{15} \geq 7/6$:
\begin{align*}
\begin{array}{rc}
C^{(iv)}_{15} = & \frac{1}{15} \big(C(1,5,6, \substack{2\\3\\4\\7}) + C(1,4, \substack{5\\6}, \substack{2\\3\\7}) + C(1,7, \substack{4\\5}, \substack{2\\3\\6}) + C(2,6, \substack{1\\3}, \substack{4\\5\\7}) + C(2,7, \substack{1\\3}, \substack{4\\5\\6})
\\  
& + \frac{1}{45} \big(C(1,4, \substack{2\\5}, \substack{3\\6\\7}) + C(1,4, \substack{3\\5}, \substack{2\\6\\7}) \big) 
+ \frac{1}{15}\big(C(3,6, \substack{1\\2}, \substack{4\\5\\7}) + C(3,7, \substack{1\\2}, \substack{4\\5\\6}) \big) 
 + \frac{4}{45} C(2,3, \substack{4\\5}, \substack{1\\6\\7}) 
\\
&+\frac{1}{6} C(1,2,4, \substack{3\\5\\6\\7}) 
 + \frac{17}{90} \big(C(1,2,5,\substack{3\\4\\6\\7}) + C(1,6,7, \substack{2\\3\\4\\5}) \big)
  + \frac{1}{30} C(1,3,4, \substack{2\\5\\6\\7}) 
   + \frac{6}{15} C(1,5, \substack{2\\3}, \substack{4\\6\\7}) 
\\
& + \frac{1}{90} \big(C(1,4, \substack{2\\3}, \substack{5\\6\\7}) + C(1,6, \substack{5\\7}, \substack{2\\3\\4}) + C(1,7, \substack{4\\6}, \substack{2\\3\\5}) + C(1, \substack{2\\3}, \substack{4\\7}, \substack{5\\6}) \big) 
 + \frac{7}{45} C(1,4,5, \substack{2\\3\\6\\7}) 
\\
& + \frac{2}{45} \big(C(2, \substack{1\\3}, \substack{4\\6}, \substack{5\\7}) + C(2, \substack{1\\3}, \substack{4\\7}, \substack{5\\6}) + C(3, \substack{1\\2}, \substack{4\\6}, \substack{5\\7}) + C(3, \substack{1\\2}, \substack{4\\7}, \substack{5\\6}) \big)
 + \frac{7}{18} C(1,4, \substack{6\\7}, \substack{2\\3\\5})
\\
& + \frac{2}{15} \big(C(3,4, \substack{1\\2}, \substack{5\\6\\7}) + C(3,5, \substack{1\\2}, \substack{4\\6\\7}) \big)
+ \frac{1}{18} C(1,3,5, \substack{2\\4\\6\\7}) 
.
\end{array}
 \end{align*}

\noindent For $c_{145} \geq 257/198$:
\begin{align*}
\begin{array}{rc}
C^{(iv)}_{145}=&  \frac{7}{99} \big( C(4,5,7, \substack{1\\2\\3\\6}) + C(2, \substack{1\\3}, \substack{4\\6}, \substack{5\\7}) + C(2, \substack{1\\3}, \substack{4\\7}, \substack{5\\6}) + C(3, \substack{1\\2}, \substack{4\\6}, \substack{5\\7}) + C(3, \substack{1\\2}, \substack{4\\7}, \substack{5\\6}) \big) 
\\
& + \frac{49}{330} C(1,2,5, \substack{3\\4\\6\\7}) 
 + \frac{23}{330} C(1,3,4, \substack{2\\5\\6\\7}) 
+ \frac{4}{99}\big(C(3,6, \substack{1\\2}, \substack{4\\5\\7}) + C(3,7, \substack{1\\2}, \substack{4\\5\\6}) \big) 
  + \frac{1}{22} C(4,5,6, \substack{1\\2\\3\\7}) 
\\
&   + \frac{5}{198} C(1,5,6, \substack{2\\3\\4\\7})  
+ \frac{4}{45} C(2,3,5, \substack{1\\4\\6\\7}) 
+ \frac{31}{495} C(2,4, \substack{1\\5}, \substack{3\\6\\7})
+ \frac{46}{495} C(3,5, \substack{1\\2}, \substack{4\\6\\7}) 
   + \frac{2}{9} C(1,6, \substack{2\\3}, \substack{4\\5\\7}) 
\\
& + \frac{1}{90} C(2,4, \substack{3\\5}, \substack{1\\6\\7})
  + \frac{53}{495} C(3,4, \substack{1\\2}, \substack{5\\6\\7}) 
   + \frac{7}{990} C(3,4, \substack{1\\5}, \substack{2\\6\\7}) 
    + \frac{103}{990} C(3,4, \substack{2\\5}, \substack{1\\6\\7}) 
\\
& + \frac{1}{18} C(3,5, \substack{1\\4}, \substack{2\\6\\7}) 
  + \frac{38}{495} C(4,6, \substack{1\\5}, \substack{2\\3\\7}) 
   + \frac{17}{330} C(4,7, \substack{1\\5}, \substack{2\\3\\6}) 
    + \frac{119}{495} C(1, \substack{2\\3}, \substack{4\\5}, \substack{6\\7}) 
\\
& + \frac{1}{33} \big( C(5,6, \substack{1\\4}, \substack{2\\3\\7})+C(5,7, \substack{1\\4}, \substack{2\\3\\6}) \big) 
 + \frac{7}{495} \big( C(4, \substack{1\\5}, \substack{2\\6}, \substack{3\\7}) + C(4, \substack{1\\5}, \substack{2\\7}, \substack{3\\6}) \big)
   + \frac{49}{198} C(1,7, \substack{4\\5}, \substack{2\\3\\6}) 
\\
&  + \frac{26}{495} C(2,4,5, \substack{1\\3\\6\\7}) 
+ \frac{4}{99} \big( C(2,5, \substack{3\\4}, \substack{1\\6\\7}) + C(2,6, \substack{1\\3}, \substack{4\\5\\7}) + C(2,7, \substack{1\\3}, \substack{4\\5\\6}) \big)+\frac{4}{33} C(1,2,4, \substack{3\\5\\6\\7}) 
.
\end{array} 
\end{align*}

\noindent For $c_{245} \geq 41/36$:
\begin{align*}
\begin{array}{rc}
C^{(iv)}_{245} = &  \frac{7}{180} \big( C(1,4, \substack{2\\3}, \substack{5\\6\\7}) + C(1,4, \substack{6\\7}, \substack{2\\3\\5}) + C(3,5, \substack{1\\2}, \substack{4\\6\\7}) + C(3, \substack{1\\2}, \substack{4\\6}, \substack{5\\7}) + C(3, \substack{1\\2}, \substack{4\\7}, \substack{5\\6}) \big) 
\\
&+\frac{8}{45} C(1,2,4, \substack{3\\5\\6\\7}) 
 + \frac{1}{9} \big( C(1,4,6, \substack{2\\3\\5\\7})  + C(1, \substack{2\\3}, \substack{4\\6}, \substack{5\\7})  + C(1, \substack{2\\3}, \substack{4\\7}, \substack{5\\6}) \big)
+ \frac{1}{90} C(1,4, \substack{3\\5}, \substack{2\\6\\7}) 
\\
& + \frac{17}{90} C(1,4,7, \substack{2\\3\\5\\6})
  + \frac{1}{6} \big( C(3,4,5, \substack{1\\2\\6\\7})  + C(2,6, \substack{4\\5}, \substack{1\\3\\7}) \big) 
   + \frac{5}{18} C(1,2, \substack{4\\5}, \substack{3\\6\\7}) 
 + \frac{1}{20} C(2, \substack{1\\3}, \substack{4\\5}, \substack{6\\7}) 
\\
& + \frac{7}{90} \big( C(1,4, \substack{3\\7}, \substack{2\\5\\6}) + C(4, \substack{1\\7}, \substack{2\\5}, \substack{3\\6}) \big) 
  + \frac{1}{18} C(2,3, \substack{4\\5}, \substack{1\\6\\7}) 
   + \frac{11}{180} C(2,5, \substack{1\\3}, \substack{4\\6\\7}) 
+ \frac{1}{36} C(2,6, \substack{1\\3}, \substack{4\\5\\7})
\\
& + \frac{1}{180} \big( C(1,4, \substack{5\\6}, \substack{2\\3\\7}) + C(1,4, \substack{5\\7}, \substack{2\\3\\6}) + C(4,7, \substack{2\\5}, \substack{1\\3\\6}) \big) 
 + \frac{11}{60} C(3, \substack{1\\2}, \substack{4\\5}, \substack{6\\7})
 + \frac{2}{45} C(1,3,4, \substack{2\\5\\6\\7}) 
\\
& + \frac{7}{36} C(2,7, \substack{1\\3}, \substack{4\\5\\6})
  + \frac{1}{10} C(3,4, \substack{1\\2}, \substack{5\\6\\7}) 
   + \frac{13}{180} \big( C(3,6, \substack{1\\2}, \substack{4\\5\\7}) + C(3,7, \substack{1\\2}, \substack{4\\5\\6}) \big)
\\
&+ \frac{14}{45} C(1,4,5, \substack{2\\3\\6\\7}) +\frac{1}{12} \big( C(4,6, \substack{2\\5}, \substack{1\\3\\7}) + C(5,6, \substack{2\\4}, \substack{1\\3\\7}) + C(5,7, \substack{2\\4}, \substack{1\\3\\6}) \big)
.
\end{array}
\end{align*}

\bibliographystyle{alpha}
\bibliography{../../../M0nBibliography}

\end{document}